\documentclass[twocolumn]{autart}    

\usepackage{graphicx}          
\usepackage{subfigure}

\usepackage{amsmath,amssymb,amsfonts}
\usepackage{algorithmic}
\usepackage{textcomp}
\usepackage{bm}
\usepackage{caption}
\usepackage{comment}
\usepackage{xcolor}
\usepackage{epstopdf}
\usepackage{mathrsfs}
\usepackage[scr=boondox]{mathalfa}
\allowdisplaybreaks[1]
\usepackage{stfloats}
\usepackage{array}
\usepackage{booktabs}

\usepackage{enumitem}
\newtheorem{theorem}{Theorem}   
\newtheorem{lemma}{Lemma}
\newtheorem{problem}{Problem}
\newtheorem{proposition}{Proposition}
\newtheorem{corollary}{Corollary}

\newtheorem{remark}{Remark}

\newtheorem{definition}{Definition}

\DeclareMathOperator{\rank}{rank}
\DeclareMathOperator{\im}{im}

\newcommand{\calA}{\mathcal{A}}
\newcommand{\calB}{\mathcal{B}}
\newcommand{\calD}{\mathcal{D}}

\newcommand{\calL}{\mathcal{L}}
\newcommand{\calM}{\mathcal{M}}

\newcommand{\calT}{\mathcal{T}}
\newcommand{\calU}{\mathcal{U}}

\begin{document}

\begin{frontmatter}

\title{Data-Driven Robust Model Reference Adaptive Control\\ with Parameter Convergence} 
\vspace*{-4ex}
	
\thanks[footnoteinfo]{This paper was not presented at any IFAC meeting. 
This work was partially supported by Jiangsu Provincial Scientific Research Center of Applied Mathematics grant BK20233002. 
H. J. van Waarde acknowledges financial support by the Dutch Research Council under the NWO Talent Programme Veni Agreement (VI.Veni.222.335).
Corresponding author: Simone Baldi.}

\author[BI,cyber]{Jiwei Wang}\ead{jiwei.wang@rug.nl},    
\author[math]{Simone Baldi}\ead{simonebaldi@seu.edu.cn},               
\author[BI]{Henk J. van Waarde}\ead{h.j.van.waarde@rug.nl}  

\address[BI]{Bernoulli Institute for Mathematics, Computer Science and Artificial Intelligence, University of Groningen, The Netherlands}  
\vspace{-0.15cm}
\address[cyber]{School of Cyber Science and Engineering, Southeast University, China}          
\vspace{-0.15cm}
\address[math]{School of Mathematics, Southeast University, China}     
\vspace{-0.4cm}

                                          
\begin{abstract}
This paper provides a data-driven design guaranteeing parameter convergence in model reference adaptive control (MRAC) when the to-be-controlled system is subject to process noise. 
In the context of MRAC, parameter convergence refers to ensuring convergence of the adaptive gains to a solution of the matching equations, or to an approximate solution when noise is present.
In classical MRAC, even small noise may induce parameter drift, thus lacking robustness to noise. 
Meanwhile, existing robust MRAC methods cannot ensure parameter convergence without imposing excitation conditions on data.
A key feature of the proposed framework is to ensure convergence of the adaptive gains to an approximate solution of the matching equations without relying on persistently exciting signals. Furthermore, the matching error can be explicitly characterized as a function of the noise. 
This explicit characterization allows to establish a necessary and sufficient condition on the noise characteristics under which the limit closed-loop system matrix is Hurwitz.
In the noise-free case, the proposed framework results in exact parameter convergence. 
Notably, as compared to existing methods achieving exact parameter convergence in the noise-free case, the condition on data in the proposed framework is weaker.
\vspace{-0.4cm}
\end{abstract}

\end{frontmatter}

\section{Introduction}

Control of dynamical systems with unknown parameters is a fundamental theme in control theory. 
Adaptive control theory provides several systematic frameworks to face unknown parameters by adjusting the control \mbox{gains} online via adaptive laws driven by data generated in real-time by the closed-loop system \cite{ioannou2006adaptive,narendra2012stable}. 
In the noise-free setting, it is well known that convergence of the tracking error can be achieved without conditions on the data. 
In contrast, convergence of the adaptive gains to their ideal values requires additional conditions on the data, the most classical requirement being persistence of excitation \cite{boyd1986necessary}. 
The general lack of parameter convergence has the consequence that adaptive control is not automatically robust to noise.
Adaptive laws that ensure convergent tracking error in ideal noise-free settings can perform poorly or be destabilized by the presence of noise: parameter drift is a typical manifestation \cite{rohrs1985robustness,narendra2012stable}. 
Several countermeasures have been investigated to robustify adaptive laws, either by modifying the adaptation mechanism (via leakage, dead-zone, projection, sample covariance matrix), or by imposing excitation conditions on data \cite{anderson1986stability,ioannou2012robust,mirkin2013tube,hussain2017computable,xie2018h,franco2021robust,zhao2025data}.

This study will focus on continuous-time model reference adaptive control (MRAC), a mature adaptive control framework \cite{annaswamy2021historical,ioannou2006adaptive,landau2011adaptive,narendra2012stable,nguyen2018model}, yet still well investigated in various switched, distributed, partial-feedback, and learning-based variants \cite{yuan2020lyapunov,yue2023model,song2021partial,annaswamy2023integration,kurdila2024nonparametric}.
Among the recent advances of MRAC, results based on concurrent learning \cite{chowdhary2010concurrent,chowdhary2013concurrent,lee2019concurrent}, combined adaptation \cite{lavretsky2009combined,roy2018combined,cho2018composite}, and data informativity \cite{wang2025bridging} are especially relevant to our study because they provide, in a noise-free setting, excitation conditions for parameter convergence weaker than persistence of excitation.
In the context of MRAC, parameter convergence refers to ensuring convergence of the adaptive gains to a solution of the so-called matching equations, or to an approximate solution when noise is present.
In parallel, offline data-driven control for continuous-time systems with noisy data and without parametric models has recently received significant attention \cite{eising2025sampling,Bosso2025,wakaiki2025data,wakaiki2026data}. 
However, existing data-driven control results are primarily limited to offline schemes, and it remains unclear how to achieve similar robustness in online adaptive schemes or how noise fundamentally affects parameter convergence.
MRAC, with its inherent online adaptation mechanism, provides a natural framework to address these challenges.

A key question in robust MRAC theory is to characterize how noise affects convergence of the adaptive gains. 
In a noise-free (discrete-time) MRAC setting, necessary and sufficient conditions have been given for the existence of an adaptive law guaranteeing the convergence of the associated gains to a solution of the matching equations \cite{wang2025bridging}. 
However, the framework used to derive such conditions only holds in a noise-free setting.
Despite the progress in MRAC, a framework is lacking that can systematically explain how noise affects the parameter convergence and how this influences the closed-loop behavior.
The first open problem solved by this study is to design an MRAC mechanism guaranteeing that, even in the presence of process noise, the adaptive gains converge to an approximate solution of the matching equations. 
Notably, the mechanism does not rely on persistently exciting signals and the matching error can be explicitly characterized depending on the noise. 
The design we propose combines real-time data generated online by the closed-loop system with a set of previously recorded data, a setup in line with many concurrent learning, combined adaptation, and data informativity methods that also exploit real-time and recorded data.

A second important question in robust MRAC theory is to guarantee whether, in the presence of noise, the adaptive gains converge to a closed-loop system matrix that is Hurwitz.
Parameter drift phenomena reveal that this is not guaranteed in adaptive control \cite{ioannou2012robust}. 
Based on the characterization of the matching error as a function of noise, the second result of this study is to give conditions on the level of noise such that the closed-loop system matrix is Hurwitz in the limit.
To achieve this, we draw on the recently developed data informativity framework \cite{vanDBLSCT2025}---particularly on its extensions to model reference control \cite{wang2025necessary}---to quantify the impact of noise on parameter convergence.
The main tool underlying the proposed robust MRAC framework is a novel analysis based on quadratic matrix inequalities (QMIs) to characterize the effect of the signal-to-noise ratio on stability.
As a final result, we show that the proposed framework has desirable properties even in the noise-free setting, where exact parameter convergence can be achieved with weaker conditions on data than the relaxed excitation conditions proposed in the literature.

In summary, this work provides a data-driven robust MRAC design that can quantify the effects of process noise on convergence of its adaptive gains.
The main contributions of this work are as follows:
\begin{itemize}
	\item 
	We propose an adaptive law and associated control law that, without requiring persistently exciting signals, ensure exponential convergence of the adaptive gains to a manifold where the parameter matching error between the closed-loop system and the reference model can be explicitly characterized as a function of the noise (see Theorem~\ref{Tec}).
	\item
	By combining the Kalman-Yakubovich-Popov (KYP) lemma with QMI techniques, we provide a necessary and sufficient condition such that, for a given signal-to-noise ratio, the closed-loop system matrix is guaranteed to be Hurwitz in the limit (see Theorem~\ref{Tstability}).
	Based on the stability analysis for the closed-loop system matrix, boundedness of all closed-loop signals is proven (see Theorem~\ref{Tbound}).
	\item
	In the noise-free setting, our framework guarantees exact convergence of the adaptive gains to a solution of the matching equations under a weaker condition on data than the relaxed excitation conditions proposed in the literature. 
	In particular, it is proven that it is not necessary for the data to be such that the system dynamics can be uniquely identified (see Proposition~\ref{Pweaker}).
\end{itemize}

The remainder of this paper is organized as follows.
Section~2 reviews the MRAC setup and formulates the problem.
Section~3 presents and analyzes the robust MRAC framework.
Section~4 numerically validates the theoretical results.
Section~5 concludes the paper.

\emph{Notation:}
We denote the space of real (resp. complex) $ n $-dimensional vectors by $ \mathbb{R}^n $ (resp. $ \mathbb{C}^n $), the space of real $ n \times m $ matrices by $ \mathbb{R}^{n\times m} $, and the space of real \emph{symmetric} (resp. complex \emph{Hermitian}) $ n \times n $ matrices by $ \mathbb{S}^{n} $ (resp. $ \mathbb{H}^{n} $).
The set of \emph{positive integers} is denoted by $ \mathbb{N} $ and the set of non-negative real numbers is denoted by $\mathbb{R}_+$.
The \textit{Euclidean norm} of a vector $ v \in \mathbb{R}^n $ is denoted by $ |v| $.
The \textit{identity matrix} of appropriate size is denoted by $ I $.
Given a matrix $ A $, its \textit{transpose} is denoted by $ A^{\top} $,
its \textit{complex conjugate transpose} is denoted by $ A^* $,
and its \textit{Moore-Penrose pseudo-inverse} is denoted by $ A^{\dagger} $.
The matrix $ A \in \mathbb{S}^{n} $ is said to be \textit{positive definite} (denoted by $ A>0 $) if $ x^{\top}Ax>0 $ for all nonzero $ x\in\mathbb{R}^n $, and \textit{positive semidefinite} (denoted by $ A\geq0 $) if $ x^{\top}Ax\geq0 $ for any $ x\in\mathbb{R}^n $.
A square matrix is \textit{Hurwitz} if all its eigenvalues have negative real parts.

\section{Problem Formulation}\label{SPF}

Consider the continuous-time system
\begin{equation}\label{cs}
\dot{x}(t)=A_\textrm{s}x(t)+B_\textrm{s}u(t)+E_\textrm{s}w(t), \quad x(0)=x_0,
\end{equation}
and the continuous-time reference model
\begin{equation}\label{cr}
\dot{x}_\textrm{m}(t)=A_\textrm{m}x_\textrm{m}(t)+B_\textrm{m}r(t), \quad x_\textrm{m}(0)=x_\textrm{m0},
\end{equation}
where $ x \in \mathbb{R}^n $ is the system state, $ u \in \mathbb{R}^m $ is the control input, $ w \in \mathbb{R}^q $ is the unknown bounded process noise, $ x_\textrm{m} \in \mathbb{R}^n $ is the reference state, and $ r\in \mathbb{R}^p $ is the reference input, where $ m \geq p $.
The system matrices $ A_\textrm{s} \in \mathbb{R}^{n\times n}, B_\textrm{s} \in \mathbb{R}^{n\times m}, E_\textrm{s} \in \mathbb{R}^{n\times q} $ are unknown, whereas the reference model matrices $ A_\textrm{m} \in \mathbb{R}^{n\times n} $ and $B_\textrm{m} \in \mathbb{R}^{n\times p} $ are given. 
We assume that $ A_\textrm{m} $ is Hurwitz and $(A_\textrm{m},B_\textrm{m})$ is controllable.

The model reference control problem is to find a fixed-gain controller
\begin{equation}\label{controllerf}
u(t) = Kx(t) + Lr(t),
\end{equation}
where the gains $ K\in \mathbb{R}^{m\times n} $ and $ L\in \mathbb{R}^{m\times p} $ are such that
\begin{equation}\label{mc}
A_\textrm{m}-A_\textrm{s}-B_\textrm{s}K=0, \quad B_\textrm{m}-B_\textrm{s}L=0.
\end{equation}
The rationale for \eqref{mc}, commonly referred to as matching equations in the literature \cite{ioannou2006adaptive,narendra2012stable}, is the following: 
substituting \eqref{controllerf} into \eqref{cs} and making use of \eqref{mc}, the dynamics of the closed-loop system become
\begin{equation}\label{dymref}
\dot{x}(t)=A_\textrm{m}x(t)+B_\textrm{m}r(t)+E_\textrm{s}w(t).
\end{equation}
Via the matching equations \eqref{mc}, the closed-loop dynamics \eqref{dymref} match the dynamics of the reference model \eqref{cr}, in the sense that the tracking error $e=x-x_\textrm{m}$ has dynamics
$
\dot{e}(t)= A_\textrm{m} e(t)+E_\textrm{s}w(t).
$
This implies that the trajectories of $e$ will be bounded, with a bound depending on the bound on the process noise, but independent of the initial conditions.

Solving \eqref{mc} requires knowledge of $ A_\textrm{s} $ and $ B_\textrm{s} $. 
One way of dealing with the case that $A_s$ and $B_s$ are unknown is model reference adaptive control (MRAC), where the fixed-gain controller \eqref{controllerf} is replaced by the adaptive-gain controller
\begin{equation}\label{controller}
u(t) = \hat{K} (t)x(t) + \hat{L}(t)r(t).
\end{equation} 
Here the adaptive gains $ \hat{K}(t) \in \mathbb{R}^{m \times n} $ and $ \hat{L}(t) \in \mathbb{R}^{m \times p} $ can be interpreted as estimates of $K$ and $L$ in \eqref{controllerf}, to be updated online.
In this work, the MRAC problem will be solved by making use of data collected from system \eqref{cs}.
Suppose that measurements of an input-state trajectory $(u,x)$ of \eqref{cs} on a time interval $[0,\calT]$ are initially available.
Define the matrices
\begin{equation}\label{fisd}
\begin{aligned}
X^\textrm{D}:=&\begin{bmatrix}
x^\textrm{df}(\tau_1)& x^\textrm{df}(\tau_2)& \cdots& x^\textrm{df}(\tau_N)
\end{bmatrix}
\in\mathbb{R}^{n\times N},\\
X:=&\begin{bmatrix}
x^\textrm{f}(\tau_1)& x^\textrm{f}(\tau_2)& \cdots& x^\textrm{f}(\tau_N)
\end{bmatrix}
\in\mathbb{R}^{n\times N},\\
U:=&\begin{bmatrix}
u^\textrm{f}(\tau_1)& u^\textrm{f}(\tau_2)& \cdots& u^\textrm{f}(\tau_N)
\end{bmatrix}\in\mathbb{R}^{m\times N},\\
W:=&\begin{bmatrix}
w^\textrm{f}(\tau_1)& w^\textrm{f}(\tau_2)& \cdots& w^\textrm{f}(\tau_N)
\end{bmatrix}\in\mathbb{R}^{q\times N},
\end{aligned}
\end{equation}
where $ 0<\tau_1<\tau_2<\dots<\tau_N\leq \calT  $ and $x^\textrm{df}$, $x^\textrm{f}$, $u^\textrm{f}$ and $w^\textrm{f}$ satisfy the (low-pass filtering) equations
\begin{align}
\label{xd}
&\dot{x}^\textrm{df}(t)=-\rho{x}^\textrm{df}(t)+\dot{x}(t), 
&&x^\textrm{df}(0)=0,\\
\label{xf}
&\dot{x}^\textrm{f}(t)=-\rho{x}^\textrm{f}(t)+x(t), 
&& x^\textrm{f}(0)=0,\\
\label{uf}
&\dot{u}^\textrm{f}(t)=-\rho{u}^\textrm{f}(t)+u(t), 
&& u^\textrm{f}(0)=0,\\
\label{wf}
&\dot{w}^\textrm{f}(t)=-\rho{w}^\textrm{f}(t)+w(t), 
&& w^\textrm{f}(0)=0,
\end{align}
where $ \rho>0 $.
We note that $ {x}^\textrm{df} $ can be computed without measuring $ \dot{x} $ since, using integration by parts, the solution to \eqref{xd} is
\begin{equation}\label{xdf}
{x}^\textrm{df}(t) = x(t)-\exp(-\rho t)x(0)-\rho{x}^\textrm{f}(t).
\end{equation}
Hence, the dynamical systems \eqref{xf}-\eqref{xdf} play the role of filters to avoid direct measurements of $ \dot{x} $ \cite{wang2025experiment}, so that the data matrices $\left( U,X,X^\textrm{D}\right) $ can be obtained from measurements of $(u,x)$.
We also note that, $w$ being unknown, the values in $W$ are unknown as well.
It follows from \eqref{cs} and \eqref{xd}-\eqref{wf} that
\begin{equation}\label{fcs}
x^\textrm{df}(t)=A_\textrm{s}x^\textrm{f}(t)+B_\textrm{s}u^\textrm{f}(t)+E_\textrm{s}w^\textrm{f}(t),
\end{equation}
for all $t \in [0,\calT]$, which implies
\begin{equation}\label{fcsd}
X^\textrm{D}=A_\textrm{s}X+B_\textrm{s}U+E_\textrm{s}W.
\end{equation}

In the noise-free case ($q = 0$), the MRAC problem is to find $\hat{K}(t)$ and $\hat{L}(t)$ in \eqref{controller}, in such a way that
\begin{equation}\label{lime}
\lim\limits_{t\to\infty}
\begin{bmatrix}
	A_\textrm{m}-A_\textrm{s}-B_\textrm{s}\hat{K}(t) &\ B_\textrm{m}-B_\textrm{s}\hat{L}(t)
\end{bmatrix}=0.
\end{equation}
In the presence of noise, however, such a strong convergence result is no longer possible in general. 
This motivates the following problem.

\begin{problem}\label{Pc}
	Consider the reference model \eqref{cr} and a set of filtered data \eqref{fisd} generated by the system \eqref{cs} and filters \eqref{xf}-\eqref{xdf}.
	Find functions $ F(\cdot), G_1(\cdot) $ and $ G_2(\cdot) $ such that, for any initial condition $x(0),x_\textnormal{m}(0) \in \mathbb{R}^n$ and \mbox{$\Psi(0)\in \mathbb{R}^{2n\times(n+p)}$}, the adaptive law
	\begin{align}\nonumber
	&\dot{\Psi}(t)=G_1\left( \Psi(t),X^\textnormal{D},X,U,x^\textnormal{df}(t),x^\textnormal{f}(t),u^\textnormal{f}(t),A_\textnormal{m},B_\textnormal{m}\right)\!,\\
	\label{al2}
	&
	\begin{bmatrix}
		\hat{K}(t) & \hat{L}(t)
	\end{bmatrix}=G_2\left( \Psi(t),X^\textnormal{D},X,U,x^\textnormal{df}(t),x^\textnormal{f}(t),u^\textnormal{f}(t)\right)\!,
	\end{align}
	is such that
	\begin{equation}\label{noisyconverge}
	\lim\limits_{t\to\infty}
	\begin{bmatrix}
		A_\textnormal{m}-A_\textnormal{s}-B_\textnormal{s} \hat{K}(t) &\ B_\textnormal{m}-B_\textnormal{s}\hat{L}(t)
	\end{bmatrix} = F(W).
	\end{equation}
	In addition, provide conditions under which the closed-loop system matrix \mbox{$ \lim\limits_{t\to\infty} A_\textnormal{s}+B_\textnormal{s}\hat{K}(t) $} is Hurwitz.
\end{problem}



We emphasize that the convergence in \eqref{noisyconverge} does not, by itself, guarantee closed-loop stability. 
In fact, even though $A_m$ is Hurwitz, the limit of $A_s + B_s \hat{K}(t)$ may not be, since the residual $F(W)$ may be nonzero.
Thus, in addition to a precise characterization of $F(W)$, Problem~\ref{Pc} requires conditions that link this residual to closed-loop stability.

\section{Main Results}

In this section, we present a solution to Problem~\ref{Pc}.
We first propose an adaptive law and characterize the limit parameter matching error between the closed-loop system and the reference model (Section~\ref{Spc}).
Then, we establish necessary and sufficient conditions under which given noise characteristics guarantee that the closed-loop system matrix is Hurwitz in the limit (Section~\ref{Ssa}).
Next, we analyze input-to-state stability of the tracking error dynamics (Section~\ref{Sba}).
Finally, we specialize our results to the noise-free setting (Section~\ref{Sscns}).

\subsection{Robust parameter convergence}
\label{Spc}

Consider $t_1,t_2,\dots,t_M$ satisfying $0=t_0<t_1<t_2<\dots<t_M<t_{M+1}=\infty$ with $M\in\mathbb{N}$,
and the data matrices
\begin{equation}\label{fisdM}
	\begin{aligned}
		X_j^\textrm{D}:=&\begin{bmatrix}
			X^\textrm{D}&x^\textrm{df}(t_1)& x^\textrm{df}(t_2)& \cdots& x^\textrm{df}(t_j)
		\end{bmatrix}\in\mathbb{R}^{n\times (N+j)},\\
		X_j:=&\begin{bmatrix}
			X&x^\textrm{f}(t_1)& x^\textrm{f}(t_2)& \cdots& x^\textrm{f}(t_j)
		\end{bmatrix}\in\mathbb{R}^{n\times (N+j)},\\
		U_j:=&\begin{bmatrix}
			U&u^\textrm{f}(t_1)& u^\textrm{f}(t_2)& \cdots& u^\textrm{f}(t_j)
		\end{bmatrix}\in\mathbb{R}^{m\times (N+j)},\\
		W_j:=&\begin{bmatrix}
			W&w^\textrm{f}(t_1)& w^\textrm{f}(t_2)& \cdots& w^\textrm{f}(t_j)
		\end{bmatrix}\in\mathbb{R}^{n\times (N+j)},
	\end{aligned}
\end{equation}
where $j\in\{1,2,\dots,M\}$ and $x^\textrm{df}$, $x^\textrm{f}$, $u^\textrm{f}$ and $w^\textrm{f}$ satisfy \eqref{xf}-\eqref{xdf}.
The data $X^D, X$ and $U$ were collected offline, as in Section \ref{SPF}, while the data at the times $t_1,t_2,\dots,t_j$ are collected online during operation of an adaptive controller, to be further specified below.
Similar to \eqref{fcsd}, it follows that
\begin{equation}\label{fccsd}
	X_j^\textrm{D}=A_\textrm{s}X_j+B_\textrm{s}U_j+W_j.
\end{equation}
Next, we construct two piecewise constant matrix-valued signals $\calU:\mathbb{R}_+\to\mathbb{R}^{m\times 2n}$ and $\calD:\mathbb{R}_+\to\mathbb{R}^{2n\times 2n}$
that will be used in the adaptive law. 
These matrix-valued signals are initialized as
\begin{equation}
	\calU(t)=
	\frac{1}{N} U
	\begin{bmatrix}
		X\\X^\textrm{D}
	\end{bmatrix}^{\!\top}\!, \quad
	\calD(t)=
	\frac{1}{N}
	\begin{bmatrix}
		X\\X^\textrm{D}
	\end{bmatrix}\!
	\begin{bmatrix}
		X\\X^\textrm{D}
	\end{bmatrix}^{\!\top}\!,
\end{equation}
for $t\in [0,t_{1})$, and are recursively updated as
\begin{align}
	\label{calU}
	&\calU(t)=
		\frac{N\!+\!j\!-\!1}{N\!+\!j}\calU(t_{j-1})+
		\frac{1}{N\!+\!j}
		u^\textrm{f}(t_j)
		\begin{bmatrix}
			x^\textrm{f}(t_j)\\x^\textrm{df}(t_j)
		\end{bmatrix}^{\!\top\!}, 
	\\
	\label{calD}
	&\calD(t)=
		\frac{N\!+\!j\!-\!1}{N\!+\!j}\calD(t_{j-1})+
		\frac{1}{N\!+\!j}
		\begin{bmatrix}
			x^\textrm{f}(t_j)\\x^\textrm{df}(t_j)
		\end{bmatrix}\!
		\begin{bmatrix}
			x^\textrm{f}(t_j)\\x^\textrm{df}(t_j)
		\end{bmatrix}^{\!\top\!},
\end{align}
for $t\in [t_{j},t_{j+1})$ with $j\in\{1,2,\dots,M\}$.
Hence, $\calU(t)$ and $\calD(t)$ can be updated with the real-time data generated online by the closed-loop system, for a control input to be designed.
We note that, for any $t\geq0$, $\calD(t)$ is positive semidefinite. Such sample covariance matrices play a key role in the data-driven control literature \cite{song2024role,chiuso2025harnessing,zhao2025regularization}.

We now consider the adaptive law
\begin{equation}\label{alp}
\dot{\Psi}(t)=
-\Gamma
\left(\calD(t)\Psi(t)-R_\textrm{m}\right),
\end{equation}
where $
\Psi(t)\in \mathbb{R}^{2n\times(n+p)}
$, $ \Gamma\in \mathbb{S}^{2n} $ is an adaptation rate satisfying $ \Gamma>0 $, and
\begin{equation*}
	R_\textrm{m}=
	\begin{bmatrix}
		I&0\\A_\textrm{m}&B_\textrm{m}
	\end{bmatrix}.
\end{equation*}
Next, we show that the adaptive law in \eqref{alp} leads to a solution to Problem~\ref{Pc}.
Denote
\begin{equation*}
		\bar{W}_0=\frac{1}{N}E_\textrm{s}W
		\begin{bmatrix}
			X\\X^\textrm{D}
		\end{bmatrix}^{\!\top} \text{and }
		\bar{W}_j=\frac{1}{N+j}E_\textrm{s}W_j
		\begin{bmatrix}
			X_j\\X_j^\textrm{D}
		\end{bmatrix}^{\!\top\!} 
\end{equation*}
with $j\in\{1,2,\dots,M\}$.
We provide the following convergence result.
\begin{theorem}\label{Tec}
	Consider the data in \eqref{fisdM} generated by \eqref{cs}, \eqref{xf}, \eqref{uf} and \eqref{xdf}, with the input
	\begin{equation}\label{input}
	u(t)=\calU(t)\Psi(t)
	\begin{bmatrix}
	x(t)\\r(t)
	\end{bmatrix},
	\end{equation}
	and $ \Psi(t) $ updated by the adaptive law \eqref{alp}.
	If
	\begin{equation}\label{image}
	\im R_\textnormal{m}\subseteq \im\calD(t_M),
	\end{equation}
	then for any initial conditions $ x(0) $, $ x_\textnormal{m}(0) $ and $ \Psi(0) $, the control gains
	\begin{equation}\label{gain}
	\begin{bmatrix}
	\hat{K}(t)&\hat{L}(t)
	\end{bmatrix}:=
	\calU(t)\Psi(t)
	\end{equation}
	are such that 
	\begin{equation}\label{lim}
	\lim\limits_{t\to\infty}\!
	\begin{bmatrix}
		A_\textnormal{m}\!-\!A_\textnormal{s}\!-\!B_\textnormal{s}\hat{K}(t) & B_\textnormal{m}\!-\!B_\textnormal{s}\hat{L}(t) 
	\end{bmatrix}\!
	= \bar{W}_M \calD(t_M)^\dagger R_\textnormal{m}.
	\end{equation}
\end{theorem}
\begin{proof}
	Denote
	$
	\tilde{\Psi}(t)=\Psi(t)-\calD(t)^\dagger R_\textrm{m},
	$
	and consider the function
	\begin{equation}\label{lf}
	V(t) = \frac{1}{2}\text{tr} \left(\tilde{\Psi}(t)^{\!\top}\calD(t)\tilde{\Psi}(t)\right) \geq 0,
	\end{equation}
	where $\mathcal{D}(t)$ is positive semidefinite and the inequality in \eqref{lf} holds with equality if and only if $ \tilde{\Psi}(t)\in\ker\calD(t) $.
	By \eqref{alp} and \eqref{image}, we have that for $t\geq t_{M}$,
	\begin{equation*}
	\dot{V}(t)=-\text{tr}\left(
	\tilde{\Psi}(t)^{\!\top}
	\calD(t_M)\Gamma\calD(t_M)
	\tilde{\Psi}(t)\right).
	\end{equation*}
	Since $ \Gamma>0 $, there exists a sufficiently small scalar $ \alpha>0 $ such that \mbox{$ 2\Gamma-\alpha\calD(t_M)^\dagger\geq0 $}.
	Then, 
	\begin{align*}
		&2\calD(t_M)\Gamma\calD(t_M)-\alpha\calD(t_M)\calD(t_M)^\dagger\calD(t_M)\\
		=\ &2\calD(t_M)\Gamma\calD(t_M)-\alpha\calD(t_M) \geq 0, 
	\end{align*}
	implying that
	\begin{equation}\label{V}
	\dot{V}(t) \leq -\alpha V(t) \quad \forall t\geq t_M.
	\end{equation}
	Hence, $ \lim\limits_{t\to\infty}\mathcal{D}(t) \tilde{\Psi}(t) = 0 $.
	Note that by \eqref{calD}, we have
	\begin{equation*}
		\calD(t_M)=\frac{1}{N+M}
		\begin{bmatrix}
			X_M\\X_M^\textrm{D}
			\end{bmatrix}
		\begin{bmatrix}
			X_M\\X_M^\textrm{D}
		\end{bmatrix}^{\!\top\!},
	\end{equation*}
	implying that
\begin{equation}\label{limPsi}
	\lim\limits_{t\to\infty}\begin{bmatrix}
	X_M\\X_M^\textrm{D}
	\end{bmatrix}^{\!\top\!}\Psi(t)=
	\begin{bmatrix}
	X_M\\X_M^\textrm{D}
	\end{bmatrix}^{\!\top\!}\calD_M^\dagger	R_\textrm{m}.
\end{equation}
	Then, \eqref{limPsi} and \eqref{image} give
	\begin{equation}\label{limRm}
	\lim\limits_{t\to\infty}\calD(t_M)\Psi(t)=\calD(t_M)\calD(t_M)^\dagger R_\textrm{m}=R_\textrm{m}.
	\end{equation}
	Note that by \eqref{calU}, we have
	\begin{equation*}
		\calU(t_M)=\frac{1}{N+M}U_M
		\begin{bmatrix}
			X_M\\X_M^\textrm{D}
		\end{bmatrix}^{\!\top\!}.
	\end{equation*}
	Then, \eqref{fccsd} and \eqref{gain} give
\begin{align*}
	& \lim\limits_{t\to\infty}
	\begin{bmatrix}
	A_\textrm{m}\!-\!A_\textrm{s}\!-\!B_\textrm{s}\hat{K}(t) & B_\textrm{m}\!-\!B_\textrm{s}\hat{L}(t)
	\end{bmatrix}\\
	=&	\begin{bmatrix}
	A_\textrm{m}\!-\!A_\textrm{s}& B_\textrm{m}
	\end{bmatrix}\!-
	\frac{1}{N\!+\!M}B_\textrm{s}U_M\lim\limits_{t\to\infty}
	\begin{bmatrix}
	X_M\\X_M^\textrm{D}
	\end{bmatrix}^{\!\top\!} \Psi(t)
\end{align*}
\begin{align*}
	=&	\begin{bmatrix}
	A_\textrm{m}\!-\!A_\textrm{s}& B_\textrm{m}
	\end{bmatrix}\!-\\
	&\frac{1}{N\!+\!M}
	(X^\textrm{D}_M\!-\!A_\textrm{s}X_M\!-\!E_\textrm{s}W_M)
	\begin{bmatrix}
	X_M\\X_M^\textrm{D}
	\end{bmatrix}^{\!\top\!}
	\calD(t_M)^\dagger R_\textrm{m}
	\\
	=&	\begin{bmatrix}
	A_\textrm{m}\!-\!A_\textrm{s}& B_\textrm{m}
	\end{bmatrix}\!-\!
	\begin{bmatrix}
	A_\textrm{m}& B_\textrm{m}
	\end{bmatrix}\!+\!
	\begin{bmatrix}
	A_\textrm{s}& 0
	\end{bmatrix}\!+
	\bar{W}_M\calD(t_M)^\dagger R_\textrm{m}\\
	=&\ \bar{W}_M\calD(t_M)^\dagger R_\textrm{m},
\end{align*}
	which concludes the proof.
\end{proof}


The validity of Theorem~\ref{Tec} relies on the assumption in \eqref{image}.
To further discuss this assumption, the following result provides sufficient conditions under which \eqref{image} holds.
\begin{lemma}\label{Limage}
	If there exists a solution $(K,L)$ to the matching equations \eqref{mc}, and
	\begin{equation}\label{ranknmq}
	\rank \begin{bmatrix}
	X_M\\U_M\\W_M
	\end{bmatrix}=n+m+q,
	\end{equation}
	then \eqref{image} holds.
\end{lemma}
\begin{proof}
	By \eqref{ranknmq}, the following holds
	\begin{align*}
	\im \calD(t_M) & = 
	\im \begin{bmatrix}
		X_M\\X_M^\textrm{D}
	\end{bmatrix} =
	\im \begin{bmatrix}
	I&0&0\\A_\textrm{s}&B_\textrm{s}&E_\textrm{s}
	\end{bmatrix}
	\begin{bmatrix}
	X_M\\U_M\\W_M
	\end{bmatrix}
	\\
	& = 
	\im \begin{bmatrix}
	I&0&0\\A_\textrm{s}&B_\textrm{s}&E_\textrm{s}
	\end{bmatrix}.
	\end{align*}
	Since the matching equations \eqref{mc} hold for some $ K $ and $ L $, we have
	\begin{equation}\label{imageABE}
	\im R_\textrm{m}
	\subseteq
	\im \begin{bmatrix}
	I&0\\A_\textrm{s}&B_\textrm{s}
	\end{bmatrix}
	\subseteq
	\im \begin{bmatrix}
	I&0&0\\A_\textrm{s}&B_\textrm{s}&E_\textrm{s}
	\end{bmatrix}
	=
	\im \calD(t_M),
	\end{equation}
	that is, \eqref{image} holds.
\end{proof}

A few remarks follow.


\begin{remark}
	The feasibility of achieving the rank condition \eqref{ranknmq} has been analyzed and discussed for discrete-time systems in \cite{naveen2024data}.
	To relate \eqref{ranknmq} to existing excitation notions in continuous-time adaptive control, we revisit a few standard definitions.
	A signal $\bm{v}(t)$ is said to have finite excitation \cite{cho2018composite} if there exist $\tau\geq0$, $\delta>0$, and $T>0$ such that
	\vspace{0.2ex}
	\begin{equation}\label{fe}
		\int_{\tau}^{\tau+T}
		\bm{v}(t)\bm{v}(t)^{\!\top}dt \geq \delta I,
	\end{equation}
	and is said to have persistent excitation if there exist $\delta>0$ and $T>0$ such that
	\begin{equation}\label{pe}
		\int_{\tau}^{\tau+T}
		\bm{v}(t)\bm{v}(t)^{\!\top}dt \geq \delta I,
		\quad \forall \tau\geq0.
	\end{equation}
	We note that \eqref{ranknmq} implies finite excitation of the signal
	\begin{equation}\label{v}
		\bm{v}=\left[ 
		\begin{aligned}
			x^\textnormal{f\ }\\
			u^\textnormal{f\ }\\
			w^\textnormal{f\ }
		\end{aligned}
		\right],
	\end{equation}
	but not persistent excitation. 
	Furthermore, even when $\bm{v}$ in \eqref{v} is not persistently exciting, condition \eqref{ranknmq} may still hold.
\end{remark}

\begin{remark}
	We note that Theorem~\ref{Tec} does not require the matching equations \eqref{mc} to hold.
	In fact, even when the matching equations \eqref{mc} fail to hold, \eqref{image} can still be satisfied for a wide range of scenarios.
	Indeed, \eqref{imageABE} reveals that in the case that \eqref{ranknmq} holds, the noise matrix $ E_\textnormal{s} $ can contribute to satisfy the image condition \eqref{image}.
	For example, in the special case that $ \rank(E_\textnormal{s})=n $, \eqref{ranknmq} implies
	$$
	\im R_\textnormal{m} \subseteq \mathbb{R}^{2n}
	\!=
	\im\! \begin{bmatrix}
	I&0&0\\A_\textnormal{s}&B_\textnormal{s}&E_\textnormal{s}
	\end{bmatrix}
	\!=
	\im\! \begin{bmatrix}
	X_M\\X_M^\textnormal{D}
	\end{bmatrix}
	\!=
	\im \calD(t_M),
	$$ 
	that is, \eqref{image} holds.
\end{remark}

%
%
%
%
%

\subsection{Stability analysis}
\label{Ssa}

Theorem~\ref{Tec} accomplishes the design objective i) of Problem~\ref{Pc}.
Addressing the stability objective ii) constitutes the main focus of this section.
Before delving into the analysis, let us recall the Kalman-Yakubovich-Popov (KYP) lemma \cite[Lemma 2.11]{scherer2000linear}.
\begin{lemma}\label{LKYP}
	Let 
	$$ \calA\in\mathbb{R}^{r\times r}, \ \calB\in\mathbb{R}^{r\times s},\  \calM:=
	\begin{bmatrix}
	\calM_{11}&\calM_{12}\\\calM_{21}&\calM_{22}
	\end{bmatrix}\in\mathbb{S}^{r+s} $$
	 be given.
	There exists $ P\in \mathbb{S}^r $ such that
	\begin{equation}
	\begin{bmatrix}
	I&0\\\calA&\calB
	\end{bmatrix}^{\!\top\!}
	\begin{bmatrix}
	0&P \\ P&0
	\end{bmatrix}
	\begin{bmatrix}
	I&0\\\calA&\calB
	\end{bmatrix}
	+\calM<0
	\end{equation} 
	if and only if $ \calM_{22}\!<\!0 $ and, for all $ \omega \!\in\! \mathbb{R} $ and $ v\!\in\!\mathbb{C}^{r+s}\backslash\{0\} $,
	\begin{equation}\label{iomega}
	\begin{bmatrix}
	\calA-i\omega I & \calB
	\end{bmatrix}v=0
	\Rightarrow
	v^*\calM v<0,
	\end{equation}
	where $ i $ is the imaginary unit such that $ i^2=-1 $.
\end{lemma}

A few auxiliary results are now given. 
To this purpose, let us define, for any $ Q\in \mathbb{R}^{n\times n} $, $ R\in \mathbb{R}^{n\times n} $ and $ c\in \mathbb{C} $,
\begin{equation}\label{Lambda}
\Lambda(Q,R,c):=Q-(R-c I)^*(R-c I).
\end{equation}
The following result allows to conclude non-singularity of \eqref{Lambda} for all complex numbers $c$ on the imaginary axis.

\begin{lemma}\label{imcirc}
	$ \Lambda(Q,R,i\omega) $ is non-singular for all $ \omega\in \mathbb{R} $ if and only if 
	\begin{equation*}
	\Theta(Q,R) := \begin{bmatrix}
	0&I\\
	R^\top R-Q&R- R^\top
	\end{bmatrix}
	\end{equation*}
	has no eigenvalues on the imaginary axis.
\end{lemma}
\begin{proof}
	\emph{Sufficiency}:\\
	Suppose that $ \Theta(Q,R) $ has no eigenvalues on the imaginary axis.
	For any $ \omega\in \mathbb{R} $, let $ v_1\in\mathbb{C}^n $ be such that $ \Lambda(Q,R,i\omega)v_1=0 $.
	Denote $ v_2:=i\omega v_1 $.
	Then, 
	\begin{align*}
	\Lambda(Q,R,i\omega)v_1
	&=\left( -\omega^2I - i\omega(R\!-\!R^\top) + Q-R^{\top\!}R\right) v_1\\
	&=i\omega v_2-(R\!-\!R^\top)v_2-(R^{\top\!} R\!-\!Q)v_1=0,
	\end{align*}
	implying that 
	\begin{equation}\label{lambda}
	(R^\top R-Q)v_1+(R-R^\top)v_2=i\omega v_2.
	\end{equation}
	Denote $ v=[v_1^\top,v_2^\top]^\top $.
	Then, by $ v_2=i\omega v_1 $ and \eqref{lambda}, we have $ \Theta(Q,R) v=i\omega v $.
	Since $ \Theta(Q,R) $ has no eigenvalue on the imaginary axis, we have $ v=0 $, and hence $ v_1=0 $.
	Then, it follows that $ \Lambda(Q,R,i\omega) $ is non-singular.
	Since $ \omega $ can take any value in $ \mathbb{R} $, we conclude that $ \Lambda(Q,R,i\omega) $ is non-singular for all $ \omega\in \mathbb{R} $.
	
	\emph{Necessity}:\\
	Suppose that $ \Lambda(Q,R,i\omega) $ is non-singular for all $ \omega\in \mathbb{R} $.
	For any $ \omega\in \mathbb{R} $, let $ v\in\mathbb{C}^{2n} $ be such that $ \Theta(Q,R) v=i\omega v $.
	Denote $ v=[v_1^\top,v_2^\top]^\top $ with $ v_1,v_2\in\mathbb{C}^n $.
	Then, we have $ v_2=i\omega v_1 $ and \eqref{lambda} holds.
	By substituting $ v_2=i\omega v_1 $ into \eqref{lambda}, we obtain
	$$
	0=(-\omega^2I - i\omega(R-R^\top) + Q-R^\top R)v_1=\Lambda(Q,R,i\omega)v_1.
	$$
	Since $ \Lambda(Q,R,i\omega) $ is non-singular for all $ \omega\in \mathbb{R} $, it follows that $ v_1=0 $, and hence $ v=0 $.
	Therefore, $ i\omega $ is not an eigenvalue of $ \Theta(Q,R) $.
	Since $ \omega $ can take any value in $ \mathbb{R} $, we conclude that $ \Theta(Q,R) $ has no eigenvalues on the imaginary axis.
\end{proof}

We then recall a result about QMIs.
Define the set
\begin{equation*}
\mathcal{Z}_{r}(\Pi):=\left\{Z\in \mathbb{R}^{r \times s}\ \Bigg| \begin{bmatrix}
I\\Z
\end{bmatrix}^{\top}\Pi
\begin{bmatrix}
I\\Z
\end{bmatrix}\geq0\right\},
\end{equation*}
where 
\begin{equation}\label{pi}
\Pi=\begin{bmatrix}
\Pi_{11}&\Pi_{12}\\
\Pi_{21}&\Pi_{22}
\end{bmatrix}\in \mathbb{S}^{s+r}.
\end{equation}
Denote the (generalized) Schur complement of $\Pi$ with respect to $\Pi_{22}$ as $ \Pi|\Pi_{22}=\Pi_{11}-\Pi_{12}\Pi^{\dagger}_{22}\Pi_{21} $.
Then, we define the set
\begin{equation*}
\mathbf{\Pi}_{s,r}\!=\!\left\{\Pi\!\in\!\mathbb{S}^{s+r}\!\mid\!\Pi_{22}\!\leq\! 0, \ker\Pi_{22}\!\subseteq\!\ker\Pi_{12}, \Pi|\Pi_{22}\!\geq\!0\right\}\!.
\end{equation*}
We state the following result instrumental to bridge the two sets we just defined.
\begin{lemma}\label{LZ}
	(\cite[Lemma 4.5]{van2023quadratic}).
	Let $ \calM\in \mathbf{\Pi}_{s,r} $. 
	Let \mbox{$ x\in \mathbb{R}^s $} and $ y\in \mathbb{R}^r $ be vectors, with $ x $ nonzero, such that 
	$$ \begin{bmatrix}
	x\\y
	\end{bmatrix}^{\!\top\!} M \begin{bmatrix}
	x\\y
	\end{bmatrix}\geq0. $$
	Then, there exists
	$
	Z \in \mathcal{Z}_{r}(\calM) $ 
	such that $ y= Zx.
	$
\end{lemma}

Finally, we define
\begin{equation*}
\Pi(Q,R):=
\begin{bmatrix}
Q-R^\top R&\quad R^\top \\ 
R&-I
\end{bmatrix}
\in \mathbf{\Pi}_{n,n},
\end{equation*}
and consider the following result which plays an essential role in stability analysis.
\begin{lemma}\label{Lstability}
	Assume $ \Theta(Q,R) $ has no eigenvalues on the imaginary axis, $ Q\geq0 $ and $ R $ is Hurwitz.
	Then, the following two statements are equivalent:
	\begin{itemize}
		\item[a)] $ Z\in \mathcal{Z}_{n}(\Pi(Q,R)) \Rightarrow Z $ is Hurwitz;
		\item[b)] $ \Lambda(Q,R,0)<0 $.
	\end{itemize}
\end{lemma}
\begin{proof}
	\underline{a) $\Rightarrow$ b)}:\\
	Suppose that any $ Z\in \mathcal{Z}_{n}(\Pi(Q,R)) $ is Hurwitz.
	We aim to prove that $ \Lambda(Q,R,0)<0 $. 
	Suppose on the contrary that $ \Lambda(Q,R,0)\not<0 $, that is,
	$$ \begin{bmatrix}
	I \\ R
	\end{bmatrix}^*
	\begin{bmatrix}
	Q & \ 0 \\ 0 & -I
	\end{bmatrix}
	\begin{bmatrix}
	I \\ R
	\end{bmatrix}\not<0. $$
	Then, there exists a nonzero vector $ x \in \mathbb{R}^n $ such that
	\begin{align*}
	&\begin{bmatrix}
	x \\ 0
	\end{bmatrix}^*
	\begin{bmatrix}
	I & \ 0\\ R & -I
	\end{bmatrix}^{\!\top\!}
	\begin{bmatrix}
	Q & \ 0 \\ 0 & -I
	\end{bmatrix}
	\begin{bmatrix}
	I & \ 0\\ R & -I
	\end{bmatrix}
	\begin{bmatrix}
	x \\ 0
	\end{bmatrix}
	\\
	=&\begin{bmatrix}
	x \\ 0
	\end{bmatrix}^*
	\Pi(Q,R)
	\begin{bmatrix}
	x \\ 0
	\end{bmatrix}\geq 0.
	\end{align*}
	Since $\Pi(Q,R)\in \mathbf{\Pi}_{n,n}$, we obtain from Lemma~\ref{LZ} that there exists $ Z\in \mathbb{R}^{n\times n} $ satisfying
	$
	Z \in \mathcal{Z}_{r}\left(\Pi(Q,R)\right)
	$
	and such that
	\begin{equation}\label{P2}
	\begin{bmatrix}
	x\\0
	\end{bmatrix}= \begin{bmatrix}
	I\\ Z
	\end{bmatrix}x.
	\end{equation}
	However, from (\ref{P2}) it follows that $ Z $ has a zero eigenvalue and is therefore not Hurwitz.
	This contradicts the initial assumption.
	We thus conclude that $ \Lambda(Q,R,0)<0 $ when $ Z $ is Hurwitz for any $ Z\in \mathcal{Z}_{n}(\Pi(Q,R)) $.
	\\
	\underline{b) $\Rightarrow$ a)}:\\
	Suppose $ \Lambda(Q,R,0)<0 $.
	Since $ \Theta(Q,R) $ has no eigenvalues on the imaginary axis, we have that $ \Lambda(Q,R,i\omega) $ is non-singular for all $ \omega\in \mathbb{R} $ according to Lemma~\ref{imcirc}.
	In addition, since $ Q\geq0 $, we have $ \Lambda(Q,R,i\omega) \in \mathbb{H}^{n} $, so that the eigenvalues of $ \Lambda(Q,R,i\omega) $ are real for any $ \omega \in \mathbb{R} $.
	Meanwhile, $ \Lambda(Q,R,i\omega) $ depends continuously on $ \omega $, implying that the eigenvalues of $ \Lambda(Q,R,i\omega) $ also vary continuously with $ \omega $ \cite{kato2012short}.
	Now, suppose on the contrary that there exists $ \omega\in \mathbb{R} $ such that $ \Lambda(Q,R,i\omega)\not<0 $. 
	By the continuity of the eigenvalues of $ \Lambda $ and the fact that $ \Lambda(Q,R,0)<0 $, there must exist $\bar{\omega}\in \mathbb{R}$ for which $ \Lambda(Q,R,i\bar{\omega}) $ is singular. 
	This leads to a contradiction. 
	We conclude that $ \Lambda(Q,R,i\omega)<0 $ for all $ \omega\in \mathbb{R} $.
	Note that 
	\begin{align*}
	\Lambda(Q,R,i\omega)
	&=Q-R^\top R + i\omega R^\top - i\omega R - \omega^2 I\\
	&=
	\begin{bmatrix}
	I\\i\omega I
	\end{bmatrix}^*
	\Pi(Q,R)
	\begin{bmatrix}
	I\\i\omega I
	\end{bmatrix}.
	\end{align*}
	Then, $ \Lambda(Q,R,i\omega)<0 $ if and only if \eqref{iomega} holds for $ \calA=0 $, $ \calB=I $ and $ \calM=\Pi(Q,R) $.
	By Lemma~\ref{LKYP}, this implies that there exists $ P\in \mathbb{S}^n $ such that
	\begin{equation}\label{MP}
	\begin{bmatrix}
	0&P\\P&0
	\end{bmatrix}+\Pi(Q,R)<0.
	\end{equation}
	Then, we have
	\begin{equation*}
	\begin{bmatrix}
	I\\R
	\end{bmatrix}^{\!\top\!}
	\left(\begin{bmatrix}
	0&P\\P&0
	\end{bmatrix}+\Pi(Q,R)\right)
	\begin{bmatrix}
	I\\R
	\end{bmatrix}<0.
	\end{equation*}
	Since
	\begin{equation*}
	\begin{bmatrix}
	I\\R
	\end{bmatrix}^{\!\top\!}
	\Pi(Q,R)
	\begin{bmatrix}
	I\\R
	\end{bmatrix}=
	Q\geq0,
	\end{equation*}
	we have
	\begin{equation}
	\begin{bmatrix}
	I\\R
	\end{bmatrix}^{\!\top\!}
	\begin{bmatrix}
	0&P\\P&0
	\end{bmatrix}
	\begin{bmatrix}
	I\\R
	\end{bmatrix}<0.
	\end{equation}
	Since $ R $ is Hurwitz, there exists $ \mathcal{Q}>0 $ such that \mbox{$ R^\top P + P R = - \mathcal{Q} $}, implying $ P = \int_{0}^{\infty}e^{R^\top t}\mathcal{Q}e^{Rt}dt>0 $ \cite[Theorem 3.28]{trentelman2002control}.
	Hence, we conclude that for any $ Z\in \mathcal{Z}_{n}(\Pi(Q,R)) $, there exists $ P>0 $ such that
	\begin{equation*}
	\begin{bmatrix}
	I\\Z
	\end{bmatrix}^{\!\top\!}
	\begin{bmatrix}
	0&P\\P&0
	\end{bmatrix}
	\begin{bmatrix}
	I\\Z
	\end{bmatrix}<
	-\begin{bmatrix}
	I\\Z
	\end{bmatrix}^{\!\top\!}
	\Pi(Q,R)
	\begin{bmatrix}
	I\\Z
	\end{bmatrix}
	\leq 0,
	\end{equation*}
	implying that $ Z $ is Hurwitz.
\end{proof}

Given the above auxiliary results, we are in a position to start the analysis for the stability objective ii) of Problem~\ref{Pc}.
For notational convenience, let us define 
$$ B_{\hat{K}}:=\lim\limits_{t\to\infty}B_\textrm{s}\hat{K}(t). $$
The subsequent analysis proceeds as follows.
We first construct an ellipsoidal set around $ A_\textrm{m} $.
The set will be constructed to contain all possible $ A_\textrm{s}+B_{\hat{K}} $ arising from the adaptive law \eqref{alp} and \eqref{gain}.
Then, we will derive necessary and sufficient conditions on $ A_\textrm{m} $ and certain noise characteristics to guarantee that all matrices within the ellipsoidal set are Hurwitz.

Let us consider the set $ \mathcal{Z}_{n}(\Pi(Q_\gamma,A_\textrm{m})) $ with 
\begin{equation}\label{Q}
Q_\gamma:= \gamma^2 (I + A_\textrm{m}^\top A_\textrm{m}),
\end{equation}
where the scalar $ \gamma>0 $ is such that
\begin{equation}\label{SNR}
\bar{W}_M^\top\bar{W}_M\leq\gamma^2 \calD(t_M)^2.
\end{equation}
The inequality \eqref{SNR} can be interpreted as a \emph{signal-to-noise ratio}, governed by the scalar $\gamma>0$.
Note that such a $ \gamma $ always exists since
\begin{equation*}
\im \bar{W}_M^\top \subseteq \im \begin{bmatrix}
X_M\\X_M^\textrm{D}
\end{bmatrix} = \im \calD(t_M).
\end{equation*}
Recall from Theorem~\ref{Tec} that the control gains \eqref{gain} with $ \Psi(t) $ updated by the adaptive law \eqref{alp} are such that \eqref{lim} holds.
By \eqref{lim}, we have
$$
A_\textrm{m}-A_\textrm{s}-B_{\hat{K}}=
\bar{W}_M\calD(t_M)^\dagger
\begin{bmatrix}
I & A_\textrm{m}^\top
\end{bmatrix}^{\!\top\!}.
$$
Then, by \eqref{SNR} and \eqref{image}, we have the ellipsoidal set
\begin{equation*}
\begin{aligned}
& (A_\textrm{m}-A_\textrm{s}-B_{\hat{K}})^\top
(A_\textrm{m}-A_\textrm{s}-B_{\hat{K}})
\\
\leq&\ \gamma^2
\begin{bmatrix}
I & A_\textrm{m}^\top
\end{bmatrix}
\calD(t_M)^\dagger\calD(t_M)^2\calD(t_M)^\dagger
\begin{bmatrix}
I & A_\textrm{m}^\top
\end{bmatrix}^{\!\top\!}
\\
=&\	\gamma^2
\begin{bmatrix}
I & A_\textrm{m}^\top
\end{bmatrix}
\begin{bmatrix}
I & A_\textrm{m}^\top
\end{bmatrix}^{\!\top\!} = Q_\gamma,
\end{aligned}
\end{equation*}
which implies 
\begin{equation}\label{matchingSNR}
A_\textrm{s}+B_{\hat{K}}\in \mathcal{Z}_{n}(\Pi(Q_\gamma,A_\textrm{m})).
\end{equation}
Then, $ \mathcal{Z}_{n}(\Pi(Q_\gamma,A_\textrm{m})) $ contains all possible $ A_\textrm{s}+B_{\hat{K}} $ arising from the adaptive law \eqref{alp} and \eqref{gain}, and $ Q_\gamma $ serves as a metric to quantify how closely $ A_\textrm{s}+B_{\hat{K}} $ adheres to the reference matrix $A_\textrm{m}$.
Accordingly, the stability objective ii) of Problem~\ref{Pc} can be reformulated as the problem of finding the conditions on $ A_\textrm{m} $ and $ \gamma $ such that \eqref{matchingSNR} implies that $ A_\textrm{s}+B_{\hat{K}} $ is Hurwitz.
It turns out that a necessary and sufficient condition can be obtained for such a problem, which can be stated as follows.

\begin{theorem}\label{Tstability}
	For a given  $ \gamma > 0 $, assume $ \Theta(Q_\gamma,A_\textnormal{m}) $ has no eigenvalues on the imaginary axis.
	Consider the control gains \eqref{gain} updated by the adaptive law \eqref{alp}.
	Then,  all $ Z\in \mathcal{Z}_{n}(\Pi(Q_\gamma,A_\textnormal{m})) $ are Hurwitz if and only if 
	\begin{equation}\label{stability}
		\Lambda(Q_\gamma,A_\textnormal{m},0)<0.
	\end{equation}
\end{theorem}
\begin{proof}
	The proof directly follows from Lemma~\ref{Lstability} since $ Q_\gamma\geq0 $ and $ A_\textrm{m} $ is Hurwitz.
\end{proof}

In view of \eqref{matchingSNR}, \eqref{stability} allows to conclude that, despite $ A_\textrm{s} $ being unknown, the matrix $ A_\textrm{s}+B_{\hat{K}} $ is Hurwitz, that is, $ \hat{K} $ in \eqref{gain} converges to a stabilizing gain.
We note from the definition of $ \Lambda $ in \eqref{Lambda} that, for a Hurwitz matrix $ A_\textrm{m} $, \eqref{stability} holds for sufficiently small $ \gamma $.
This implies that convergence to a stabilizing gain requires 
a sufficiently large signal-to-noise ratio.

\subsection{Boundedness analysis}
\label{Sba}

In this subsection, we show the boundedness of all closed-loop signals and we characterize how the limit tracking error depends on the noise.
To this end, we first recall some standard notation.
A continuous function $ f: \mathbb{R}_+ \to \mathbb{R}_+ $ is of class $ \mathcal{N}_0 $ if it is nondecreasing and satisfies $ f(0)=0 $.
A function $ f $ of class $ \mathcal{N}_0 $ is also of class $ \mathcal{K} $ if it is strictly increasing.
A continuous function $ g: \mathbb{R}_+ \times \mathbb{R}_+ \to \mathbb{R}_+ $ is of class $ \mathcal{KL} $ if, for each fixed $ t\geq0 $, the mapping $ \tau\to g(\tau,t) $ is a $ \mathcal{K} $ function, and for each fixed $ \tau\geq0 $, the mapping $ t\to g(\tau,t) $ is non-increasing with $ g(\tau,t)\to0 $ as $ t\to\infty $.
For a matrix-valued signal $ S:\mathbb{R}_+\to\mathbb{R}^{r\times s} $, we define the seminorm
\begin{equation*}
\|S\|_\infty:=\text{ess$ \, $}\sup\{\|S(t)\|,t\geq0\},
\end{equation*}
where $ \|\cdot\| $ is any matrix norm.
We denote by $\calL_\infty^{r,s}$ the space of signals $S$ with $ \|S\|_\infty < \infty. $
In addition, we use the shorthand notation $\calL_\infty^r := \calL_\infty^{r,1}$.
With this notation, we recall some key definitions of input-to-state stability.

\begin{definition}\label{Diss}
	(\cite{sontag1989smooth}). Consider an input-affine system
	\begin{equation}\label{nls}
	\dot{\xi}(t)=\phi(\xi(t))+\Upsilon(\xi(t))\mu(t)
	\end{equation}
	with $ \phi:\mathbb{R}^n\to\mathbb{R}^n $ and $ \Upsilon:\mathbb{R}^n\to\mathbb{R}^{n\times m} $.
	The system \eqref{nls} is input-to-state stable (ISS) if there exist a function $ g $ of class $ \mathcal{KL} $ and a function $ f $ of class $ \mathcal{K} $ such that for every measurable essentially bounded $ \mu(\cdot) $ and every initial state $ \xi_0 $, the solution $ \xi(t) $ exists for every $ t \geq 0 $ and satisfies
	\begin{equation}\label{issc}
	|\xi(t)|\leq g(|\xi_0|,t)+f(\|\mu\|_\infty).
	\end{equation}
\end{definition}
Definition~\ref{Diss} implies that, for any initial state $ \xi_0 $ and any input $ u\in \calL_{\infty}^m $, the corresponding state satisfies $ x\in\calL_{\infty}^n $.

We aim to study input-to-state stability for the dynamics of the tracking error
\begin{equation*}
\begin{aligned}
&\ \dot{e}(t)=\dot{x}(t)-\dot{x}_\textrm{m}(t)
\\
=& \left(A_\textrm{s}\!+\!B_\textrm{s}\hat{K}(t)\right)\! x(t)+B_\textrm{s}\hat{L}(t)r(t)-A_\textrm{m}x_\textrm{m}(t)-B_\textrm{m}r(t)
\\
=&\left( A_\textrm{s}\!+\!B_\textrm{s}\hat{K}(t)\right)\! e(t)+\\
&\left( A_\textrm{s}\!+\!B_\textrm{s}\hat{K}(t)\!-\!A_\textrm{m}\right)\! x_\textrm{m}(t)+\left( B_\textrm{s}\hat{L}(t)\!-\!B_\textrm{m}\right)\! r(t).
\end{aligned}
\end{equation*}
Denote
\begin{equation}\label{hatu}
\begin{aligned}
\hat{A}(t)&:=A_\textrm{s}\!+\!B_\textrm{s}\hat{K}(t),\\
\hat{u}(t)&:=
\begin{bmatrix}
A_\textrm{s}\!+\!B_\textrm{s}\hat{K}(t)\!-\!A_\textrm{m}\
&\ B_\textrm{s}\hat{L}(t)\!-\!B_\textrm{m}
\end{bmatrix}
\begin{bmatrix}
x_\textrm{m}(t)\\r(t)
\end{bmatrix}.
\end{aligned}
\end{equation}
Then, the dynamics of $ e $ can be written as
\begin{equation}\label{e}
\dot{e}(t)=\hat{A}(t)e(t)+\hat{u}(t),\quad e(0) = x_0-x_\textrm{m0}.
\end{equation}
The following result is instrumental to analyze the dynamics \eqref{e}.

\begin{lemma}\label{Liss}
	Consider a linear time-varying system
	\begin{equation}\label{csxi}
	\dot{\xi}(t) = A(t) \xi(t) + \mu(t).
	\end{equation}
	If $ A(t) $ converges to a Hurwitz matrix $ A_H $ exponentially, then \eqref{csxi} is ISS.
\end{lemma}
\begin{proof}
	Suppose $ \lim_{t\to\infty} A(t) = A_{H} $ exponentially with $ A_H $ being Hurwitz.
	Then, there exist $ c_0>0 $ and $ \lambda_0>0 $ such that $ \|A(t) - A_{H}\| \leq c_0 \exp(-\lambda_0 t) $ for any $ t\geq0 $.
	Since $ A_H $ is Hurwitz, there exists $ P>0 $ such that \mbox{$ A_H^\top P+PA_H = -I $}.
	Furthermore, there exist $ c_1,c_2>0 $ such that $ c_1I\leq P\leq c_2I $.
	Consider the auxiliary autonomous system
	\begin{equation}\label{as}
	\dot{z}(t)=A(t)z(t),
	\end{equation}
	and the Lyapunov function $ V(z)=z^\top P z $.
	Then, we have
	\begin{align*}
	\dot{V}(z(t))=&\ z(t)^{\!\top}\left( A(t)^{\!\top} P + P A(t)\right) z(t)\\
	=&\ z(t)^{\!\top} \left( A_H^\top P + P A_H\right)^{\!\top} z(t)+\\
	&\ z(t)^{\!\top} \left( (A(t)-A_H)^\top P + P (A(t)-A_H)\right)  z(t)\\
	\leq& -|z(t)|^2 + 2\|P\| \cdot \|A(t)-A_H\| \cdot |z(t)|^2\\
	\leq& -\frac{1}{c_2}V(z(t)) + \frac{2\|P\|c_0}{c_1} \exp(-\lambda_0 t) V(z(t)),
	\end{align*}
	where the first inequality follows from Cauchy-Schwarz inequality, and the second  inequality follows from $ \|A(t) - A_{H}\| \leq c_0 \exp(-\lambda_0 t) $ and $ c_1I\leq P\leq c_2I $.
	Since $ \exp(-\lambda_0 t) $ converges to $ 0 $ exponentially, there exists $ \theta_0>0 $ such that $ 4\|P\|c_0c_2\exp(-\lambda_0 \theta_0)<c_1 $.
	Then, for any $ t\geq \theta_0 $, we have 
	\begin{equation}\label{Ve}
	\dot{V}(z(t))< -\frac{1}{2c_2}V(z(t)),
	\end{equation} 
	implying that $ z\in \calL_{\infty}^n $.
	Denote the state transition matrix for \eqref{as} as $ \Phi(\theta_2,\theta_1) $ (we recall that $ \Phi(\theta_2,\theta_1) $ is such that $z(\theta_2) = \Phi(\theta_2, \theta_1) z(\theta_1)$ for $ \theta_2\geq \theta_1>0 $).
	We consider three possible ordering cases for $\theta_0$, $\theta_1$, and $\theta_2$.
	\\
	In case $ \theta_0\geq \theta_2\geq \theta_1 $, there exist $ c_3>0 $ and $ \lambda>0 $ such that 
	$$
	\|\Phi(\theta_2,\theta_1)\|\leq c_3 \leq c_3 \exp(\lambda \theta_0) \cdot \exp(-\lambda(\theta_2-\theta_1)).
	$$
	In case $ \theta_2\geq \theta_1\geq \theta_0 $, by \eqref{Ve}, there exist $ c_4>0 $ and $ \lambda>0 $ such that
	$$
	\|\Phi(\theta_2,\theta_1)\|\leq c_4 \exp(-\lambda(\theta_2-\theta_1)).
	$$
	Finally, in case $ \theta_2\geq \theta_0\geq \theta_1 $, there exist $ c_4,c_5>0 $ and $ \lambda>0 $ such that
	\begin{align*}
	\|\Phi(\theta_2,\theta_1)\| 
	& = \|\Phi(\theta_2,\theta_0)\Phi(\theta_0,\theta_1)\| \\
	& \leq \|\Phi(\theta_0,\theta_1)\|\cdot \|\Phi(\theta_2,\theta_0)\| \\
	&\leq c_5\cdot c_4 \exp(-\lambda(\theta_2-\theta_0)) \\
	&\leq c_5 c_4 \exp(\lambda \theta_0) \cdot \exp(-\lambda(\theta_2-\theta_1)). 
	\end{align*}
	From the above three case, we conclude that there exist $ c>0 $ and $ \lambda>0 $ such that
	\begin{equation}\label{Phi}
	\|\Phi(\theta_2,\theta_1)\|\leq c \exp(-\lambda(\theta_2-\theta_1)) \text{ for } \theta_2  \geq \theta_1 > 0.
	\end{equation}
	Having considered the autonomous system \eqref{as}, we now consider the system \eqref{csxi}.
	The solution of \eqref{csxi} can be represented as 
	\begin{equation}\label{solution}
	\xi(t)=\Phi(t,0)\xi(0)+\int_0^t\Phi(t,\tau)\mu(\tau)d\tau.
	\end{equation}
	Then, we have
	\begin{align*}
	|\xi(t)|&\leq\|\Phi(t,0)\| \cdot |\xi(0)|+\int_0^t\|\Phi(t,\tau)\| \cdot |\mu(\tau)|d\tau.
	\\
	&\leq c \exp(-\lambda t) |\xi(0)| + c \|\mu\|_\infty \int_0^t\exp(-\lambda(t-\tau))d\tau
	\\
	&\leq c \exp(-\lambda t) |\xi(0)| + \frac{c}{\lambda} \|\mu\|_\infty
	\\
	&:=g(|\xi(0)|,t)+f(\|\mu\|_\infty).
	\end{align*}
	Note that $ g(s,t) $ is a $ \mathcal{KL} $ function and $ f(s) $ is a $ \mathcal{K} $ function.
	Hence, \eqref{csxi} is ISS. 
\end{proof}

We are now in a position to state the following result.

\begin{theorem}\label{Tbound}
	Assume that \eqref{image} holds and $ \Theta(Q_\gamma,A_\textnormal{m}) $ has no eigenvalues on the imaginary axis.
	For any $x(0),x_\textnormal{m}(0) \in \mathbb{R}^n$ and $\Psi(0)\in \mathbb{R}^{2n\times(n+p)}$, consider the adaptive law \eqref{alp} and the input \eqref{input}.
	Then, if \eqref{stability} holds, the resulting input-state trajectory satisfies $ x\in\calL_{\infty}^n $ and $ u\in\calL_{\infty}^m $. 
	Moreover, the tracking error satisfies $e\in\calL_{\infty}^n $ and there exists a function $ f $ of class $ \mathcal{N}_0 $ such that 
	\begin{equation}\label{eW}
	\lim_{t \rightarrow \infty} |e(t)| \leq f\left( \|\bar{W}_M\|\right) .
	\end{equation}
\end{theorem}

\begin{proof}
	Suppose that \eqref{stability} holds.
	We first show that \mbox{$ e\in\calL_{\infty}^n $}, $ x\in\calL_{\infty}^n $ and $u\in\calL_{\infty}^m $.
	Note that $ x_\textrm{m}\in\calL_{\infty}^n $ since $ r\in\calL_{\infty}^p $ and $ A_\textrm{m} $ is Hurwitz.
	Further, according to \eqref{lf} and \eqref{V}, we have $ \Psi\in \calL_{\infty}^{2n,n+p} $, which implies $ \hat{K}\in\calL_{\infty}^{m,n} $ and $ \hat{L}\in\calL_{\infty}^{m,p} $ by \eqref{gain}.
	Hence, $ \hat{u}\in\calL_{\infty}^n $.
	Next, by Theorem~\ref{Tec}, \eqref{lf} and \eqref{V} implies that $ \hat{A}(t) $ converges to $ A_\textrm{s}+B_{\hat{K}} $ exponentially.	 
	By Theorem~\ref{Tstability}, we have that $ A_\textrm{s}+B_{\hat{K}} $ is Hurwitz.
	Then, by Lemma~\ref{Liss}, we have that \eqref{e} is \mbox{ISS}, implying that $ e\in\calL_{\infty}^n $.
	From the fact that $ x_\textrm{m}\in\calL_{\infty}^n $, we further have $ x=e+x_\textrm{m}\in\calL_{\infty}^n $.
	Then, by \eqref{input}, we have $ u\in\calL_{\infty}^m $.
	\\
	We next show that \eqref{eW} holds.
	According to \cite[Sect. I-B\&C]{sontag1996new}, if \eqref{e} is ISS, then \eqref{e} satisfies the asymptotic gain property, that is, there exists a function $ \hat{f} $ of class $ \mathcal{N}_0 $ such that for any $ e(0)\in\mathbb{R}^{n} $,
	\begin{equation*}
	\limsup_{t\to\infty} |e(t)|\leq \hat{f}\left( \limsup_{t\to\infty} |\hat{u}(t)|\right).
	\end{equation*}
	By \eqref{lim} and \eqref{hatu}, we have
	\begin{equation*}
	\limsup_{t\to\infty} |\hat{u}(t)| \leq \|\bar{W}_M\| \left\|\calD(t_M)^\dagger R_\textrm{m}\right\|
	\left\|\begin{bmatrix}
	x_\textrm{m}\\r
	\end{bmatrix}\right\|_\infty.
	\end{equation*}
	Since $ x_\textrm{m}\in\calL_{\infty}^n $, $ r\in\calL_{\infty}^p $, there exists a $ \delta>0 $ such that
	\begin{equation}
	\hat{f}\left( \limsup_{t\to\infty} |\hat{u}(t)|\right)  \leq \hat{f}\left( \delta\|\bar{W}_M\|\right) :=f\left( \|\bar{W}_M\|\right).
	\end{equation}
	Hence, we conclude that the function $ f $ of class $ \mathcal{N}_0 $ is such that \eqref{eW} holds.
\end{proof}

\subsection{Noise-free setting}
\label{Sscns}

We now specialize the presented analysis to the noise-free setting, i.e., $ q=0 $ and thus $\bar{W}_M=0$.
Hence, the following holds as a consequence of Theorem~\ref{Tec}.

\begin{corollary}\label{Cnc}
	Let $ q=0 $.
	Consider the data in \eqref{fisdM} generated by \eqref{cs} with the adaptive law \eqref{alp} and the input \eqref{input}.
	Then, if \eqref{image} holds, the control gains $ \hat{K}(t) $ and $ \hat{L}(t) $ in \eqref{gain} are such that \eqref{lime} holds.
\end{corollary}

Having shown exact parameter convergence in \eqref{lime}, we next present stability and boundedness in the noise-free case.
Because stability follows directly from Corollary~\ref{Cnc}, stability and boundedness are stated jointly as a corollary.
\begin{corollary}\label{Csb}
	Let $ q=0 $.
	Assume that \eqref{image} holds.
	For any $x(0),x_\textnormal{m}(0) \in \mathbb{R}^n$ and $\Psi(0)\in \mathbb{R}^{2n\times(n+p)}$, consider the adaptive law \eqref{alp} and the input \eqref{input}.
	Then, $\hat{K}$ in \eqref{gain} converges to a stabilizing gain.
	Moreover, the resulting trajectories satisfy $ x\in\calL_{\infty}^n $, $ u\in\calL_{\infty}^m $, $e\in\calL_{\infty}^n $ and
	$
	\lim_{t \rightarrow \infty} |e(t)| =0.
	$
\end{corollary}
\begin{proof}
	By Corollary~\ref{Cnc}, \eqref{lime} holds.
	Then, $\hat{K}$ in \eqref{gain} converges to a stabilizing gain since $ A_\textrm{m} $ is Hurwitz.\\
	Since $ q=0 $, we have $ Q_\gamma=0 $ and $ \bar{W}_M=0 $.
	Since $ A_\textrm{m} $ is Hurwitz, it follows that $ \Theta(0,A_\textrm{m}) $ has no eigenvalues on the imaginary axis and $ \Lambda(0,A_\textrm{m},0)<0 $.
	Then, by Theorem~\ref{Tbound}, we have $ x\in\calL_{\infty}^n $, $ u\in\calL_{\infty}^m $, $e\in\calL_{\infty}^n $ and $ \lim_{t \rightarrow \infty} |e(t)| \leq f(0)=0 $ since the function $ f $ is of class $ \mathcal{N}_0 $.
\end{proof}

We note that noise-free MRAC approaches in the literature attain convergence to a solution of the matching equations under structural assumptions on the system, such as full-rank of $ B_\textrm{s} $, which is required to ensure that the matching equations \eqref{mc} have a unique solution \cite{chowdhary2013concurrent,roy2018combined}.
No such structural assumptions are required here.
It is also worth mentioning that existing noise-free approaches with relaxed excitation conditions on data \cite{lavretsky2009combined,chowdhary2010concurrent,chowdhary2013concurrent,cho2018composite,roy2018combined,lee2019concurrent}, typically require knowledge of the input matrix $ B_\textrm{s} $, with the exception of \cite{roy2018combined}.
There, persistence of excitation is relaxed to initial excitation, implying that a set of noise-free data can be collected such that
\begin{equation}\label{ranknm}
\rank \begin{bmatrix}
X_M\\U_M
\end{bmatrix} = n+m.
\end{equation}
From Lemma~\ref{Limage} (with \mbox{$q=0$}), we have that \eqref{ranknm} is sufficient for \eqref{image} to hold.
However, \eqref{ranknm} becomes necessary only in some special cases, meaning that \eqref{image} is weaker than \eqref{ranknm}.
This is formalized in the following result.

\begin{proposition}\label{Pweaker}
	Let $ q=0 $.
	Assume that \eqref{mc} holds for some $ K\in\mathbb{R}^{m\times n} $ and $ L\in\mathbb{R}^{m\times p} $.
	Then, for any \mbox{$M \geq 0$} and any dataset $ \left( U_M,X_M,X_M^\textnormal{D}\right) $ satisfying \eqref{image}, also \eqref{ranknm} holds if and only if $ \rank B_\textnormal{m}=m $.
\end{proposition}
\begin{proof}
	\emph{Sufficiency:}\\
	Suppose $ \rank B_\textrm{m}=m $.
	Then, we have \mbox{$ \rank R_\textrm{m}=n+m $}.
	For any \mbox{$M \geq 0$}, we consider a set of data $\left( U_M,X_M,X_M^\textrm{D}\right) $ in \eqref{fccsd} satisfying \eqref{image}.
	Then, there exists a matrix $ V\in\mathbb{R}^{2n\times (n+p)} $ such that
	\begin{equation*}
	R_\textrm{m} = \calD(t_M) V =
	\begin{bmatrix}
	I&0\\A_\textrm{s}&B_\textrm{s}
	\end{bmatrix}
	\begin{bmatrix}
	X_M \\ U_M
	\end{bmatrix}
	\begin{bmatrix}
	X_M \\ X_M^\textrm{D}
	\end{bmatrix}^{\!\top\!} V.
	\end{equation*}
	As a result,
	\begin{equation*}
	n+m= \rank R_\textrm{m} \leq \rank \begin{bmatrix}
	X_M\\ U_M
	\end{bmatrix} \leq n+m,
	\end{equation*}
	implying that \eqref{ranknm} holds.
	
	\emph{Necessity:}\\
	Suppose on the contrary that $\rank (B_\textrm{m}) \neq m$.
	We will show that there exist \mbox{$M \geq 0$} and dataset $ \left( U_M,X_M,X_M^\textnormal{D}\right) $ such that \eqref{image}, but not \eqref{ranknm}, holds.\\
	Since $p \leq m$, we have that $\rank (B_m) < m$.
	Let $(K,L)$ be a solution of the matching equations \eqref{mc}.
	We may assume without loss of generality that $\rank(L) < m$.
	Indeed, we have $\im B_\textrm{m} \subseteq \im B_\textrm{s}$ and thus $L := B_\textrm{s}^\dagger B_\textrm{m}$ satisfies the equation $B_\textrm{s} L = B_\textrm{m}$.
	Note that the matrix $L$ satisfies $\rank (L) \leq \rank (B_\textrm{m}) < m$.
	\\
	Let $ \lambda_1, \lambda_2, \dots, \lambda_n $ be the eigenvalues of $ A_\textrm{m} $ and let $ T>0 $ be such that $ (n+p)T\leq \calT $ and
	\begin{equation*}
	\lambda_{j} - \lambda_{l} \neq \frac{2a\pi}{T}i
	\end{equation*}
	holds for any distinct $ j,l \in \{1,2,\dots,n\} $ and $ a\in\mathbb{N} $.
	Then, in line with \cite[Theorem 1]{wang2025experiment}, let us design the reference signals to be piecewise constant with sampling time $T$: for any \mbox{$ t\in [jT,(j+1)T): j \in \{0,1,\dots,n+p-1\} $}, design $ r(t)=\mu_j $, where $ \mu_j\in \mathbb{R}^{p} $.
	At time $ t=0 $, select a nonzero $ \mu_0 $ such that for any $ t\in [0,T) $, $ r(t)=\mu_0 $. 
	At time $ t=jT $, for $ j \in \{1,2,\dots,n+p-1\} $,
	\begin{itemize}
		\item if
		$
		x_\textrm{m}(jT) \notin \im 
		\left[ 
		x_\textrm{m}(0)\ x_\textrm{m}(T)\ \cdots\ x_\textrm{m}((j-1)T)
		\right],
		$
		then select $ \mu_j $ arbitrarily;
		\item if 
		$
		x_\textrm{m}(jT) \in \im 
		\left[ 
		x_\textrm{m}(0)\ x_\textrm{m}(T)\ \cdots\ x_\textrm{m}((j-1)T)
		\right],
		$
		then there exist $ \eta \in \mathbb{R}^m $ and $ \xi\in \mathbb{R}^n $ with $ \eta\neq0 $ such that
		\begin{equation*}
		\begin{bmatrix}
		\xi^\top&\eta^\top
		\end{bmatrix}
		\begin{bmatrix}
		x_\textrm{m}(0)&x_\textrm{m}(T)&\cdots&x_\textrm{m}((j-1)T) \\
		r(0)&r(T)&\cdots&r((j-1)T)
		\end{bmatrix}=0.
		\end{equation*}
		In this case, select $ \mu_j $ such that
		$$
		\xi^{\top} x_\textrm{m}(jT) + \eta^{\top} \mu_j \neq 0.
		$$
	\end{itemize}
	Then, with the low-pass filters
	\begin{align}
	\label{fxm}
	\dot{x}_\textrm{m}^\textrm{f}(t)&=-\rho{x}_\textrm{m}^\textrm{f}(t)+x_\textrm{m}(t), 
	& x_\textrm{m}^\textrm{f}(0)&=0,\\
	\label{fr}
	\dot{r}^\textrm{f}(t)&=-\rho{r}^\textrm{f}(t)+r(t), 
	& r^\textrm{f}(0)&=0,
	\end{align}
	one can obtain from \cite[Theorem 1]{wang2025experiment} that the resulting filtered data satisfy
\begin{equation}\label{ranknmf}
	\rank
	\begin{bmatrix}
	x_\textrm{m}^\textrm{f}(T)& x_\textrm{m}^\textrm{f}(2T)& \cdots& x_\textrm{m}^\textrm{f}((n+p)T)\\
	r^\textrm{f}(T)& r^\textrm{f}(2T)& \cdots& r^\textrm{f}((n+p)T)
	\end{bmatrix}=n+p.
\end{equation}
	We now consider the low-pass filter
	\begin{equation}\label{fxmd}
	\dot{x}_\textrm{m}^\textrm{df}(t)=-\rho{x}_\textrm{m}^\textrm{df}(t)+\dot{x}_\textrm{m}(t), \quad
	x_\textrm{m}^\textrm{df}(0)=0.
	\end{equation}
	Then, it can be verified that for any $ j\in\{1,2,\dots,n+p\} $, $ {x}_\textrm{m}^\textrm{df}(jT)=A_\textrm{m}x_\textrm{m}^\textrm{f}(jT)+B_\textrm{m}r^\textrm{f}(jT) $.
	By denoting
	\begin{equation*}
	\begin{aligned}
	R:=&\begin{bmatrix}
	r(T)&r(2T)&\cdots&r((n+p)T)
	\end{bmatrix},\\
	X_\textrm{m}:=&\begin{bmatrix}
	x^\textrm{f}_\textrm{m}(T)&x^\textrm{f}_\textrm{m}(2T)&\cdots&x^\textrm{f}_\textrm{m}((n+p)T)
	\end{bmatrix},\\
	X^\textrm{D}_\textrm{m}:=&\begin{bmatrix}
	x^\textrm{f}_\textrm{m}(T)&x^\textrm{f}_\textrm{m}(2T)&\cdots&x^\textrm{f}_\textrm{m}((n+p)T)
	\end{bmatrix},
	\end{aligned}
	\end{equation*}
	we have 
	$$ X^\textrm{D}_\textrm{m}=A_\textrm{m}X_\textrm{m}+B_\textrm{m}R, $$ 
	implying that
	\begin{equation}\label{m}
	\begin{bmatrix}
	X_\textrm{m} \\ X^\textrm{D}_\textrm{m}
	\end{bmatrix}
	=R_\textrm{m}
	\begin{bmatrix}
	X_\textrm{m}\\R
	\end{bmatrix}.
	\end{equation}
	Therefore, \eqref{ranknmf} and \eqref{m} imply that
	\begin{equation}\label{xm}
	\im R_\textrm{m} =
	\im \begin{bmatrix}
	X_\textrm{m} \\ X_\textrm{m}^\textrm{D}
	\end{bmatrix}.
	\end{equation}
	Since $ (K,L) $ is a solution of the matching equations \eqref{mc}, we have from \eqref{m} that
	\begin{equation}
	\begin{bmatrix}
	X_\textrm{m}\\ X_\textrm{m}^\textrm{D}
	\end{bmatrix}
	=
	\begin{bmatrix}
	I&0\\A_\textrm{s}&B_\textrm{s}
	\end{bmatrix}
	\begin{bmatrix}
	I&0\\K&L
	\end{bmatrix}
	\begin{bmatrix}
	X_\textrm{m}\\R
	\end{bmatrix}.
	\end{equation}
	By defining the input $ u(t)=Kx_\textrm{m}(t)+Lr(t) $, it can be verified with the filter \eqref{uf} that
	\begin{equation*}
	U_\textrm{m}:=\begin{bmatrix}
	u(T)&u(2T)&\cdots&u((n+p)T)
	\end{bmatrix}=KX_\textrm{m}+LR.
	\end{equation*}
	The matching equations \eqref{mc} give
	$$
	\begin{aligned}
	\dot{x}(t)
	&= A_\textrm{s}x(t)+B_\textrm{s}u(t)\\
	&= A_\textrm{s}x(t)+B_\textrm{s}Kx(t)+B_\textrm{s}Lr(t)\\
	&= A_\textrm{m}x(t)+B_\textrm{m}r(t).
	\end{aligned}
	$$
	Then, $ (u,x_\textrm{m}) $ could be signals generated by \eqref{cs} under the initial condition $ x_0=x_\textrm{m0} $.
	Hence, the set of data $ (U_\textrm{m},X_\textrm{m},X_\textrm{m}^\textrm{D}) $ can be generated by system \eqref{cs} and the low-pass filters \eqref{xd}-\eqref{uf}.
	Since $ \rank (L) < m$, we have
	\begin{equation*}
	\rank\left(\begin{bmatrix}
	I&0\\K&L
	\end{bmatrix}\right)<n+m,
	\end{equation*}
	and therefore
	\begin{equation}\label{xmu}
	\rank\left(\begin{bmatrix}
	X_\textrm{m}\\U_\textrm{m}
	\end{bmatrix}\right)<n+m.
	\end{equation}
	By \eqref{xm}, we have that \eqref{image}, but not \eqref{ranknm}, holds for the data $ \left( U_M,X_M,X_M^\textrm{D}\right)  = (U_\textrm{m},X_\textrm{m},X^\textrm{D}_\textrm{m}) $ with $M=0$, resulting in a contradiction.
	Hence, we conclude that $ \rank(B_\textrm{m})=m $.
\end{proof}

\begin{remark}
	Rank conditions analogous to \eqref{ranknm} play an important role in system identification and data-driven control, where different types of data have been considered, such as samples of continuous-time trajectories \cite{de2019formulas,hu2025data}, coefficient matrices corresponding to certain basis functions \cite{rapisarda2023orthogonal}, and filtered data of continuous-time trajectories \cite{ohta2024data,wang2025experiment}.
	The condition \eqref{ranknm} allows to uniquely identify the pair $(A_\textnormal{s}, B_\textnormal{s})$ in \eqref{cs} from noise-free data \cite{wang2025experiment}. 
	Proposition~\ref{Pweaker} shows that \eqref{image} is weaker than \eqref{ranknm}, implying that \eqref{image} is weaker than requiring the collected data to uniquely identify the system.
\end{remark}




\section{Simulation}\label{Ss}

\begin{table*}[b]
	\centering
	\begin{tabular}{c|c|cc|cc}
		\toprule
		& \multicolumn{3}{c|}{Percentage of runs in which $\Lambda(Q_\gamma,A_\textnormal{m},0)<0$} & \multicolumn{2}{c}{Percentage of runs in which $ A_\textrm{s}+B_\textrm{s}\hat{K}(t) $ is Hurwitz} \\
		\midrule
		& with offline & \multicolumn{2}{c|}{with online data ($M=1200$) and} &  \multicolumn{2}{c}{$ A_\textrm{s}+B_\textrm{s}\hat{K}(t) $ evaluated at $t=30\,\mathrm{s}$ with} \\
		& data ($M=0$) & $ r(t)=[\sin(t),\cos(t)]^\top\! $ & $ r(t)=[0.1,0.1]^\top $ & \quad $ r(t)=[\sin(t),\cos(t)]^\top\! $ & $ r(t)=[0.1,0.1]^\top $ \\
		\midrule
		$\sigma=0.02$ & 100\% & 100\%  & 100\% & 100\% & 100\% \\
		$\sigma=0.06$ & 88.2\% & 100\%  & 91.4\% & 100\%  & 100\% \\
		$\sigma=0.2$ & 17.2\%  & 87.6\% & 30.2\% & 100\%  & 100\% \\
		$\sigma=0.6$ & 0.4\%  & 7.8\% & 1.2\% & 100\%  & 100\% \\
		$\sigma=2$ & 0\%  & 0\% & 0\% & 98.2\% & 98.2\% \\
		$\sigma=6$ & 0\% & 0\% & 0\%  & 88.6\% & 89\% \\
		$\sigma=20$ & 0\% & 0\% & 0\%  & 79.2\% & 78.6\%\\
		\bottomrule
	\end{tabular}
	\vspace{1ex}
	\label{table}
\end{table*}

\begin{figure*}[b]
	\centering
	\subfigure[Offline data and $ \sigma=0.06 $]
	{\label{dsinoff003}
		\includegraphics[trim=30bp 255bp 60bp 265bp, clip, height=3.5cm]{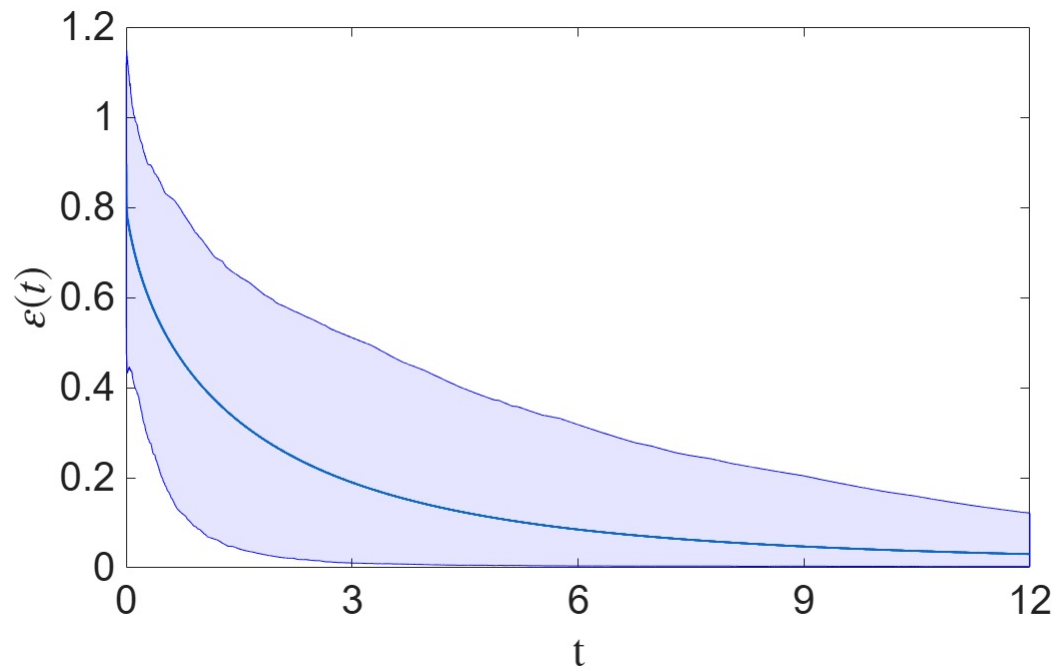} }
	\subfigure[Offline data and $ \sigma=0.2 $]
	{\label{dsinoff03}
		\includegraphics[trim=30bp 255bp 60bp 265bp, clip, height=3.5cm]{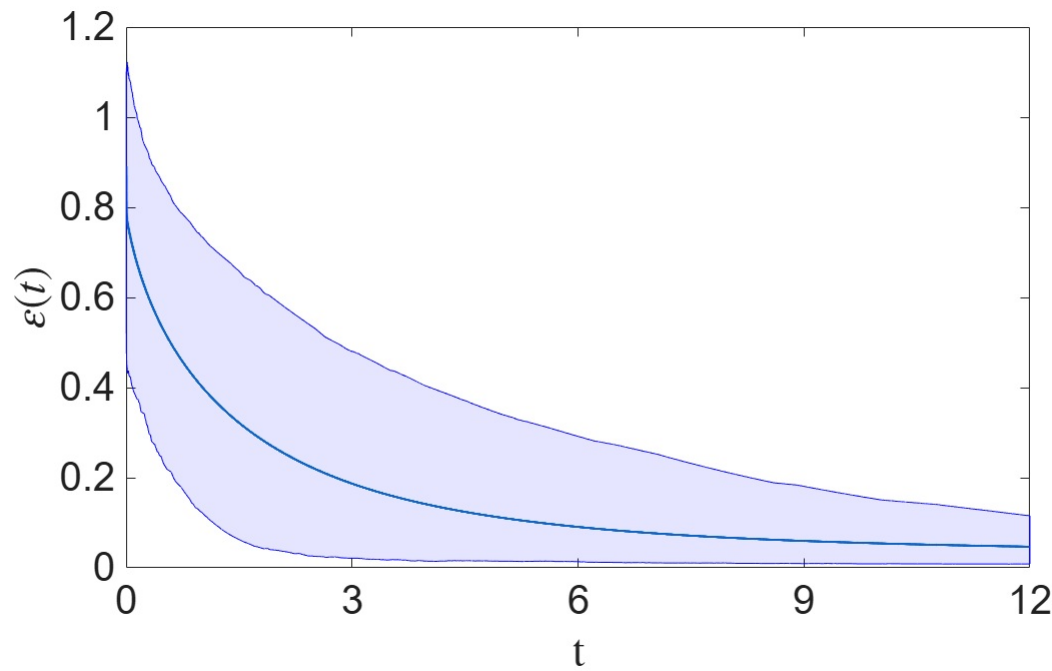} }
	\subfigure[Offline data and $ \sigma=0.6 $]
	{\label{dsinoff3}
		\includegraphics[trim=30bp 255bp 60bp 265bp, clip, height=3.5cm]{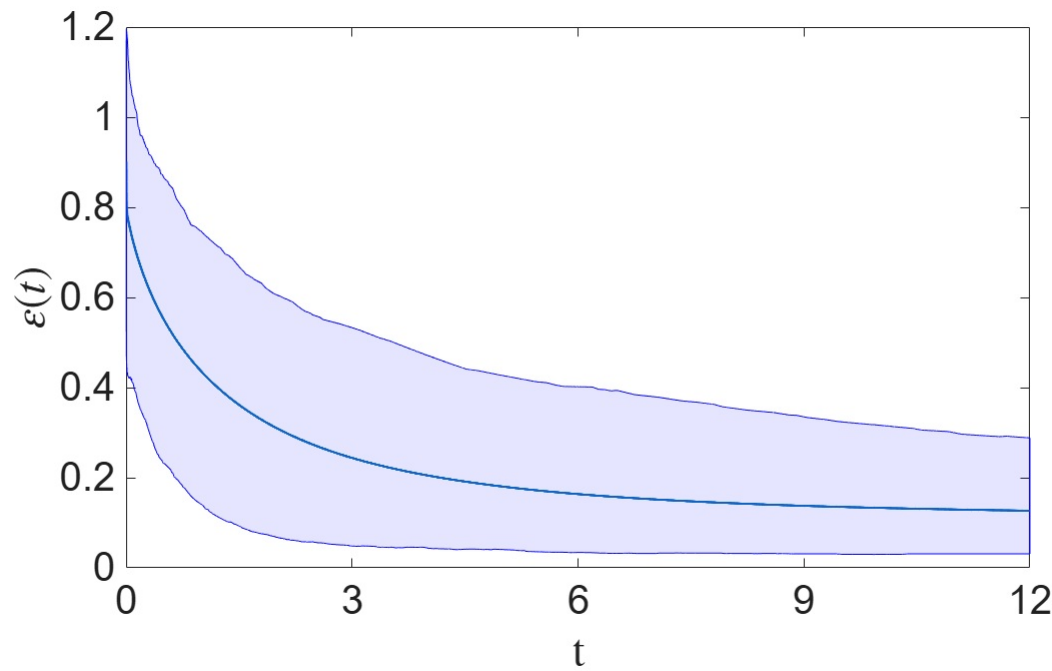} }
	\subfigure[Online data and $ \sigma=0.06 $]
	{\label{dsin003}
		\includegraphics[trim=30bp 255bp 60bp 265bp, clip, height=3.5cm]{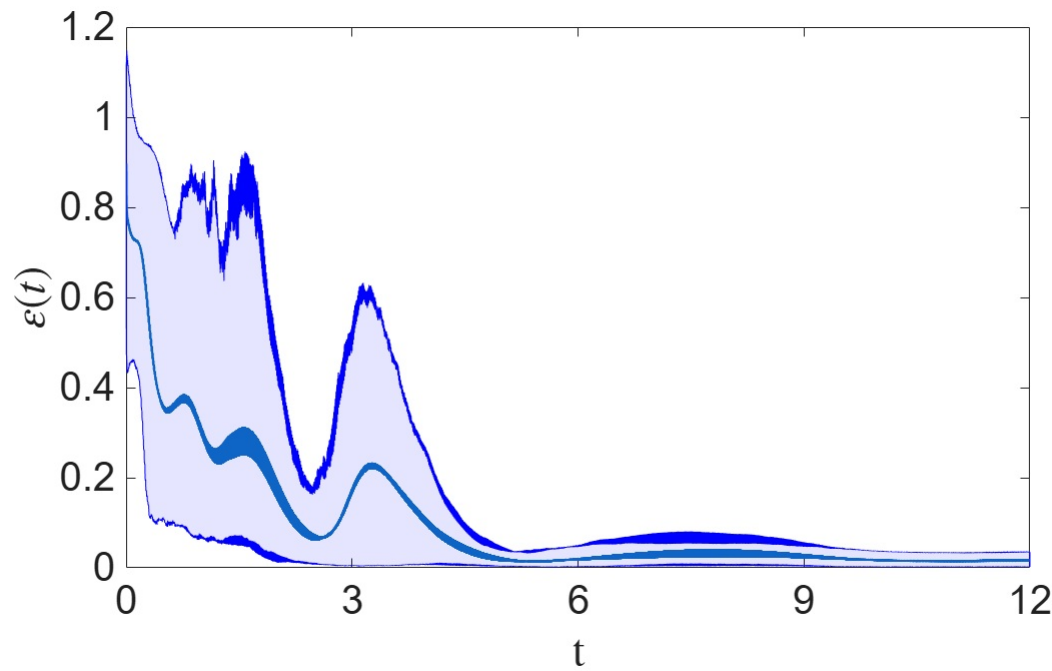} }
	\subfigure[Online data and $ \sigma=0.2 $]
	{\label{dsin03}
		\includegraphics[trim=30bp 255bp 60bp 265bp, clip, height=3.5cm]{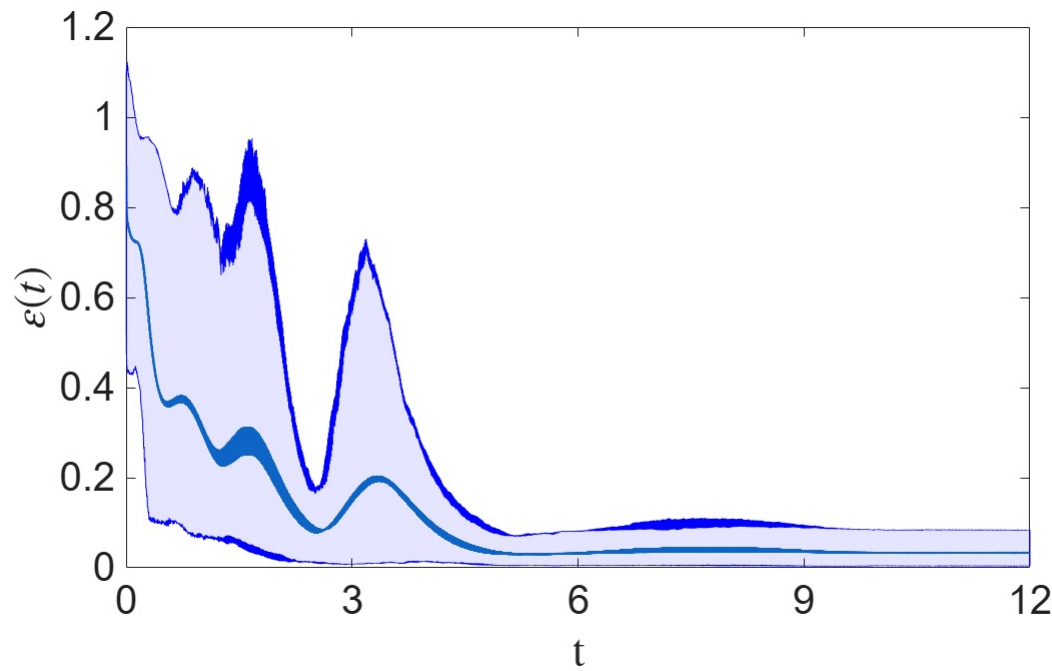} }
	\subfigure[Online data and $ \sigma=0.6 $]
	{\label{dsin3}
		\includegraphics[trim=30bp 255bp 60bp 265bp, clip, height=3.5cm]{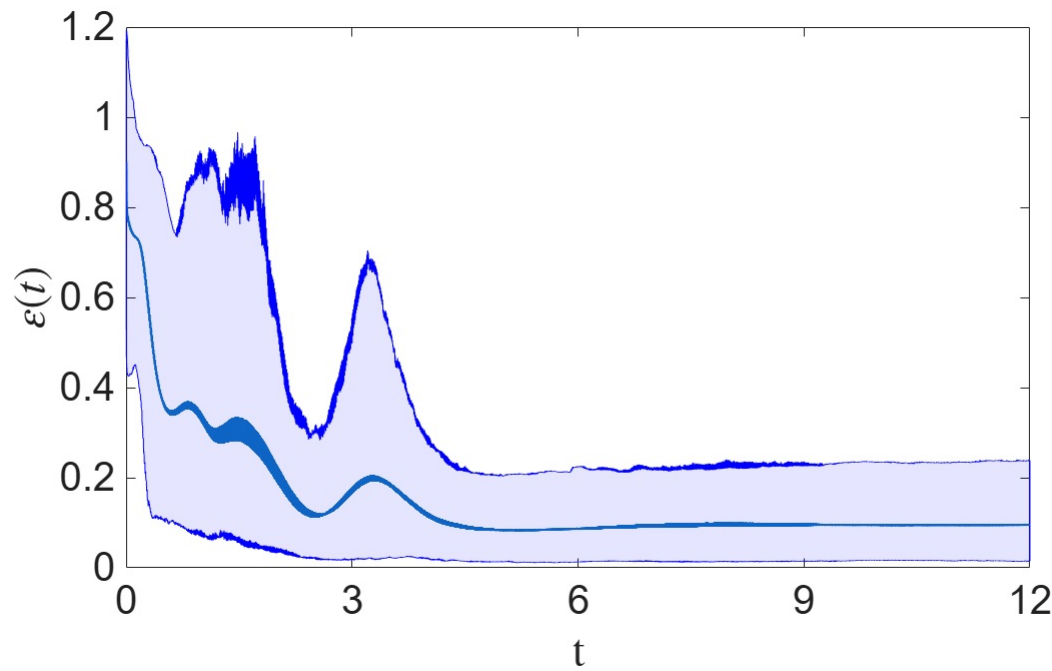}}
	\caption{The average and 90\% confidence intervals of the normalized matching errors $ \varepsilon(t) $ for $ r(t)=[\sin(t),\cos(t)]^\top $ under different noise levels ($\sigma=0.06,0.2,0.6$) and data collection (offline or online).}
	\label{Fdsin}
\end{figure*}

\begin{figure*}[b]
	\centering
	\subfigure[Offline data and $ \sigma=0.06 $]
	{\label{esinoff003}
		\includegraphics[trim=30bp 255bp 60bp 265bp, clip, height=3.5cm]{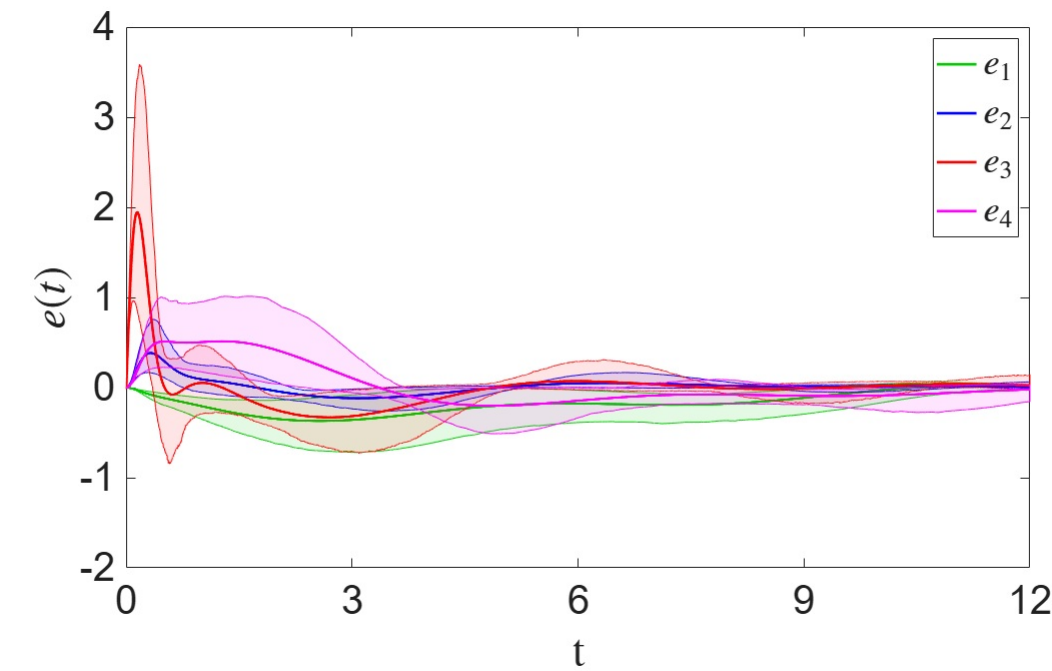} }
	\subfigure[Offline data and $ \sigma=0.2 $]
	{\label{esinoff03}
		\includegraphics[trim=30bp 255bp 60bp 265bp, clip, height=3.5cm]{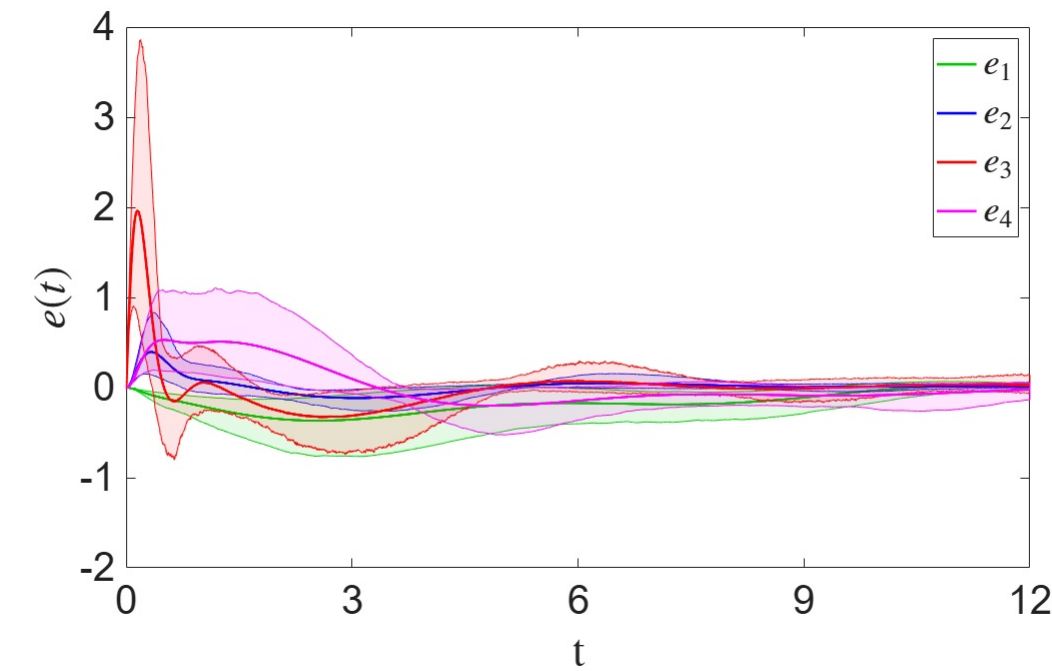} }
	\subfigure[Offline data and $ \sigma=0.6 $]
	{\label{esinoff3}
		\includegraphics[trim=30bp 255bp 60bp 265bp, clip, height=3.5cm]{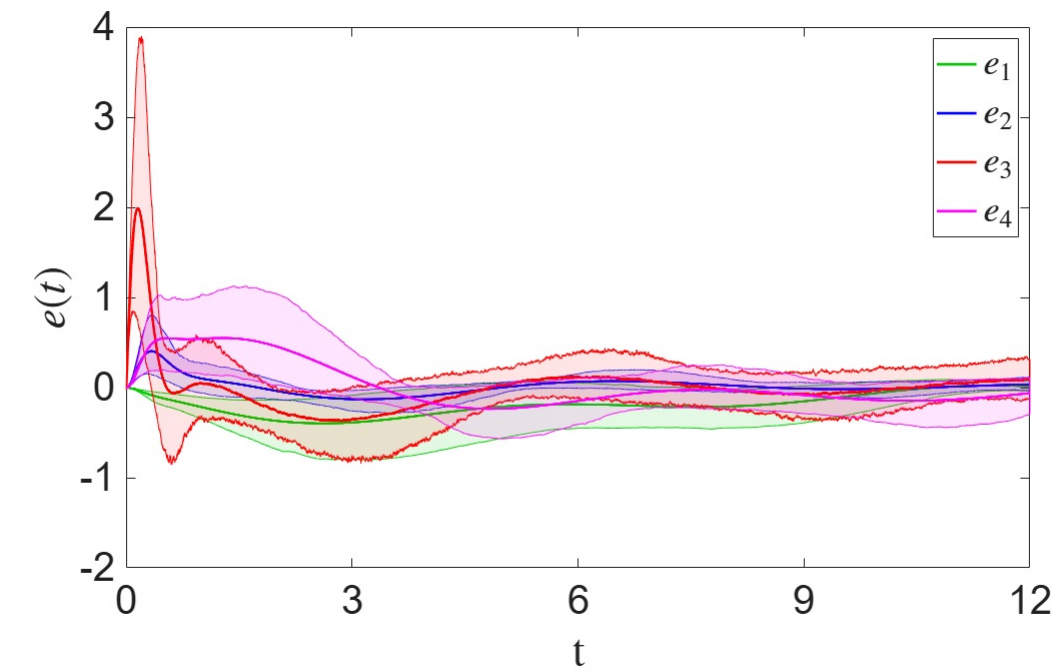} }
	\subfigure[Online data and $ \sigma=0.06 $]
	{\label{esin003}
		\includegraphics[trim=30bp 255bp 60bp 265bp, clip, height=3.5cm]{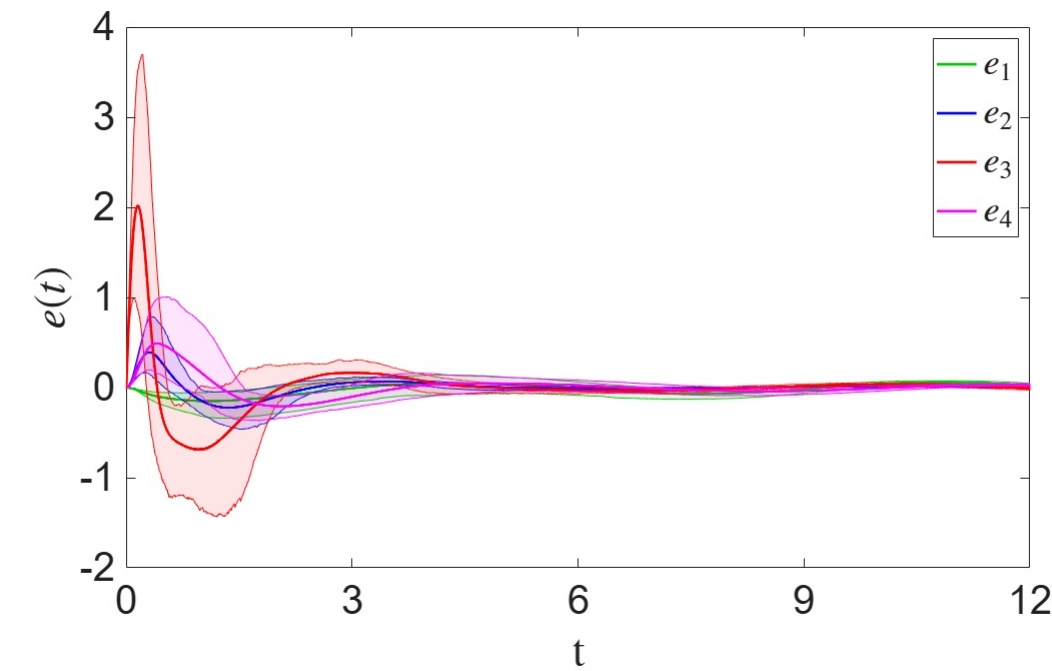} }
	\subfigure[Online data and $ \sigma=0.2 $]
	{\label{esin03}
		\includegraphics[trim=30bp 255bp 60bp 265bp, clip, height=3.5cm]{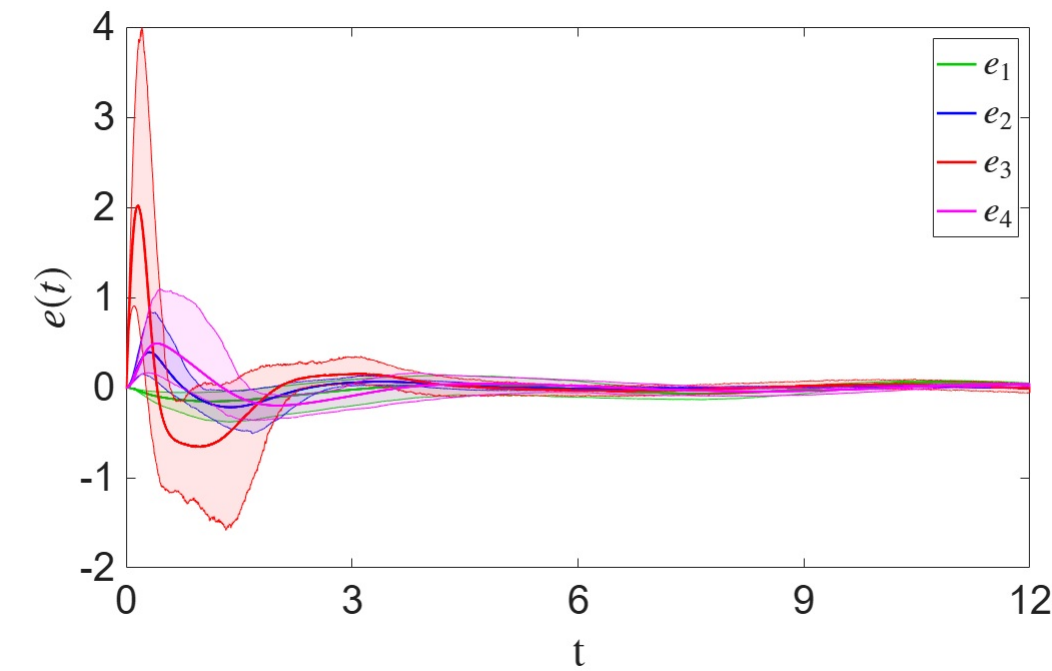} }
	\subfigure[Online data and $ \sigma=0.6 $]
	{\label{esin3}
		\includegraphics[trim=30bp 255bp 60bp 265bp, clip, height=3.5cm]{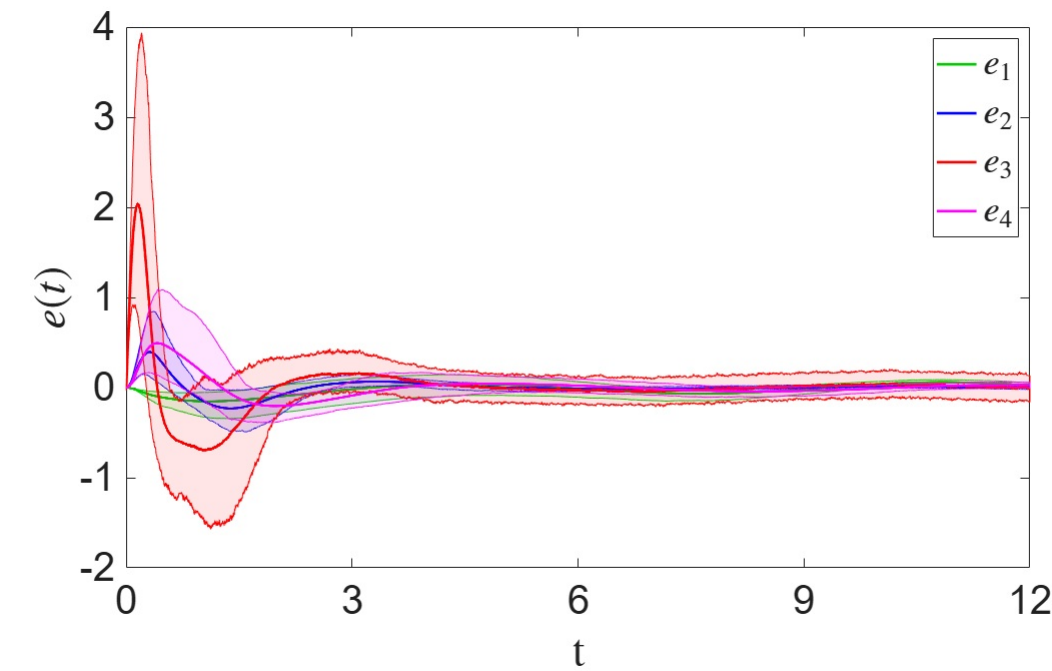} }
	\caption{The average and 90\% confidence intervals of the tracking errors $e(t)$ for $ r(t)=[\sin(t),\cos(t)]^\top $ under different noise levels ($\sigma=0.6,0.2,0.06$) and data collection (offline or online).}
	\label{Fesin}
\end{figure*}

\begin{figure*}[b]
	\centering
	\subfigure[Offline data and $ \sigma=0.06 $]
	{\label{dconoff003}
		\includegraphics[trim=30bp 255bp 60bp 265bp, clip, height=3.5cm]{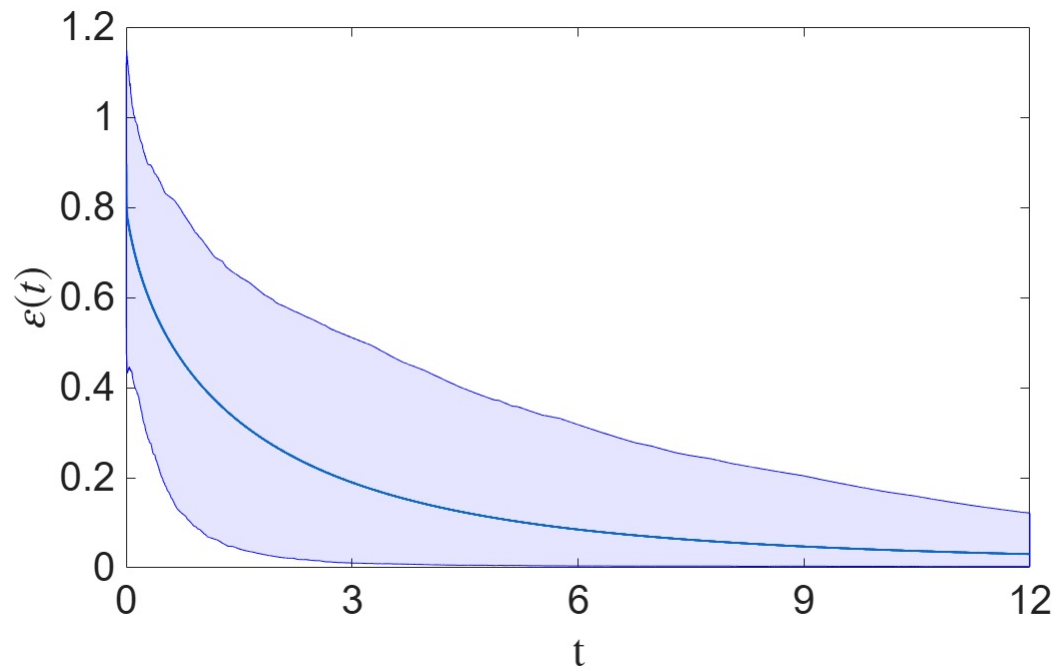} }
	\subfigure[Offline data and $ \sigma=0.2 $]
	{\label{dconoff03}
		\includegraphics[trim=30bp 255bp 60bp 265bp, clip, height=3.5cm]{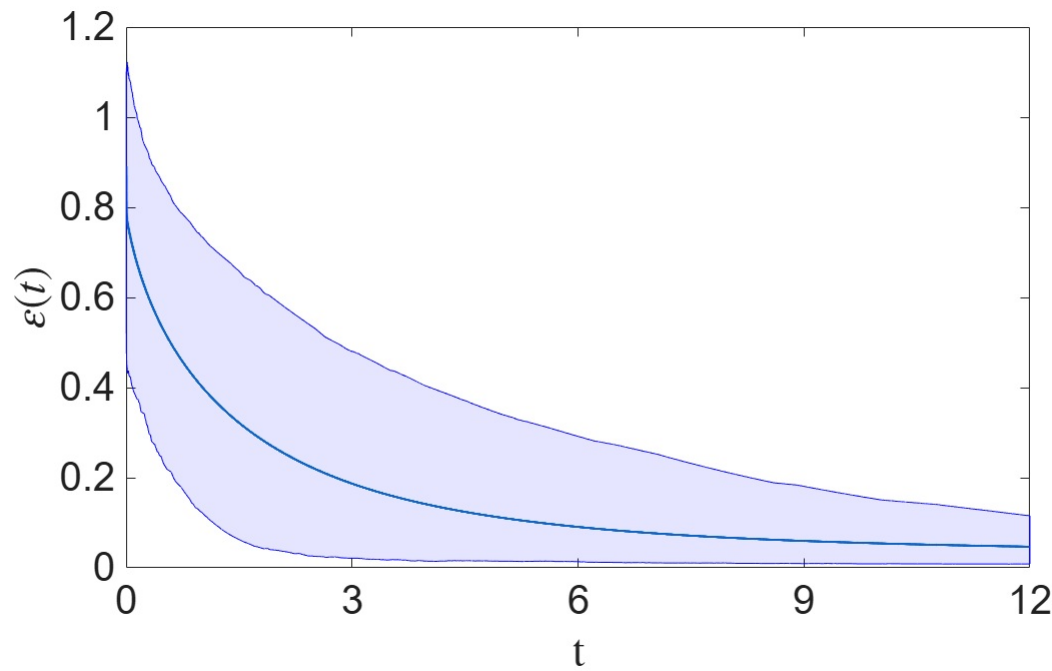} }
	\subfigure[Offline data and $ \sigma=0.6 $]
	{\label{dconoff3}
		\includegraphics[trim=30bp 255bp 60bp 265bp, clip, height=3.5cm]{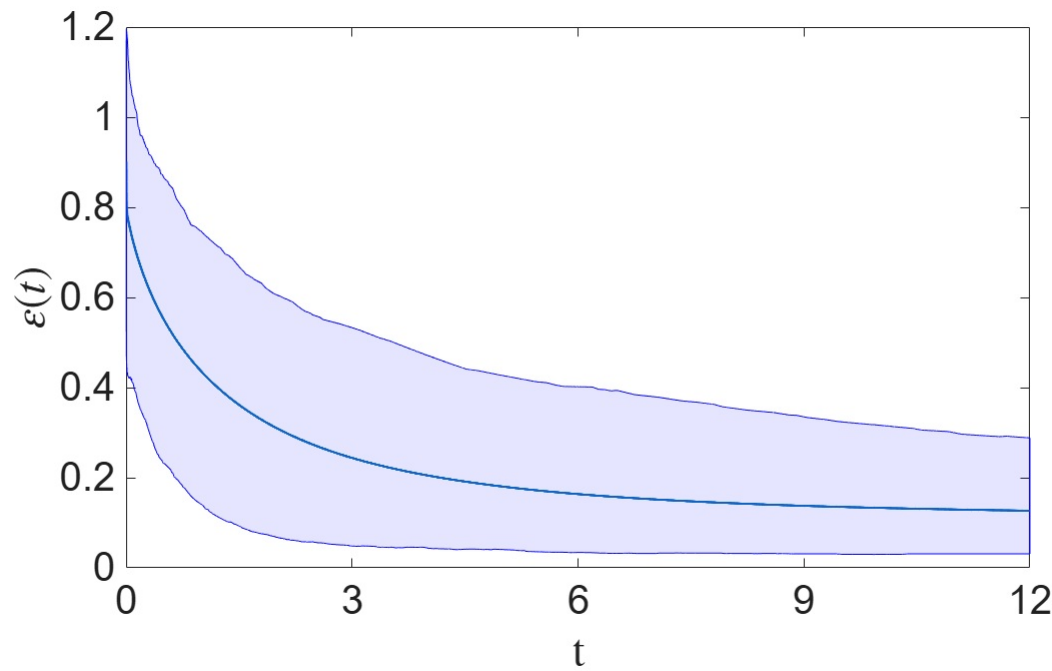} }
	\subfigure[Online data and $ \sigma=0.06 $]
	{\label{dcon003}
		\includegraphics[trim=30bp 255bp 60bp 265bp, clip, height=3.5cm]{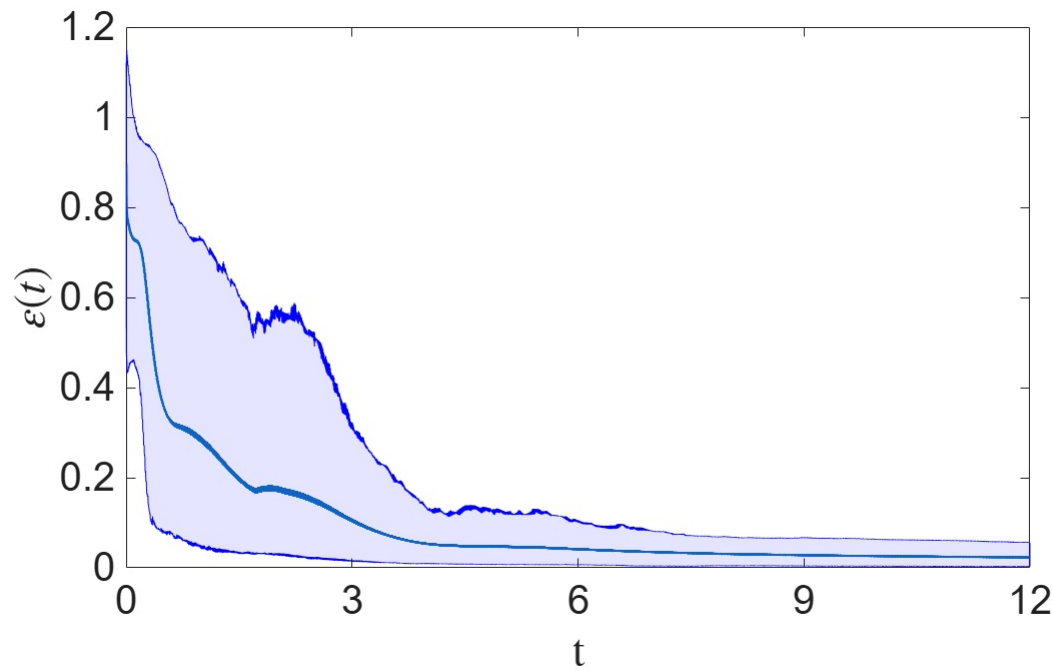} }
	\subfigure[Online data and $ \sigma=0.2 $]
	{\label{dcon03}
		\includegraphics[trim=30bp 255bp 60bp 265bp, clip, height=3.5cm]{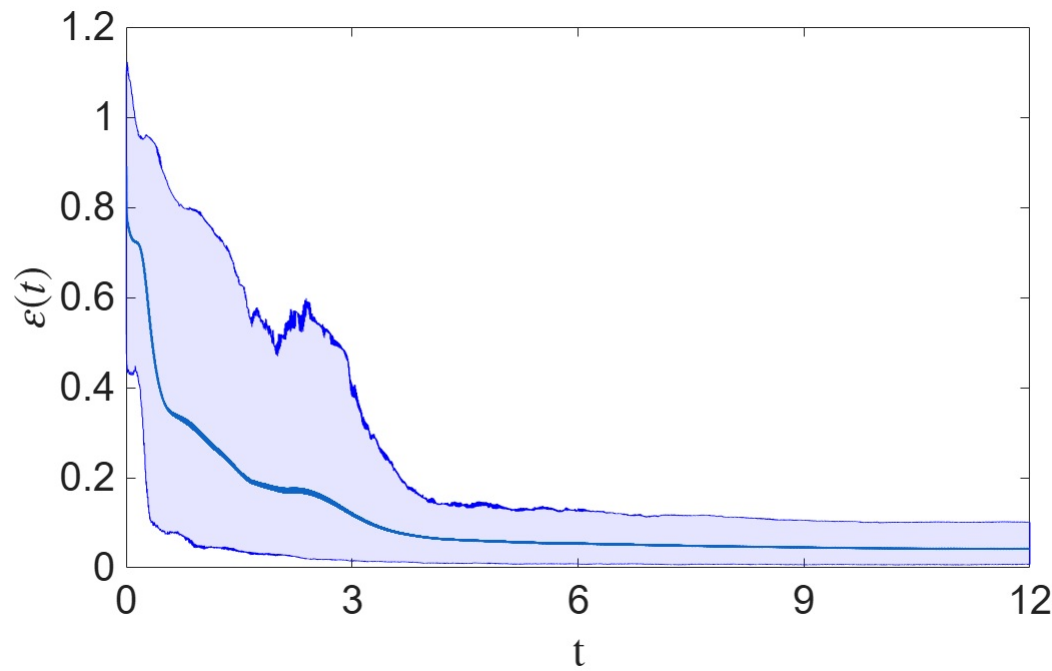} }
	\subfigure[Online data and $ \sigma=0.6 $]
	{\label{dcon3}
		\includegraphics[trim=30bp 255bp 60bp 265bp, clip, height=3.5cm]{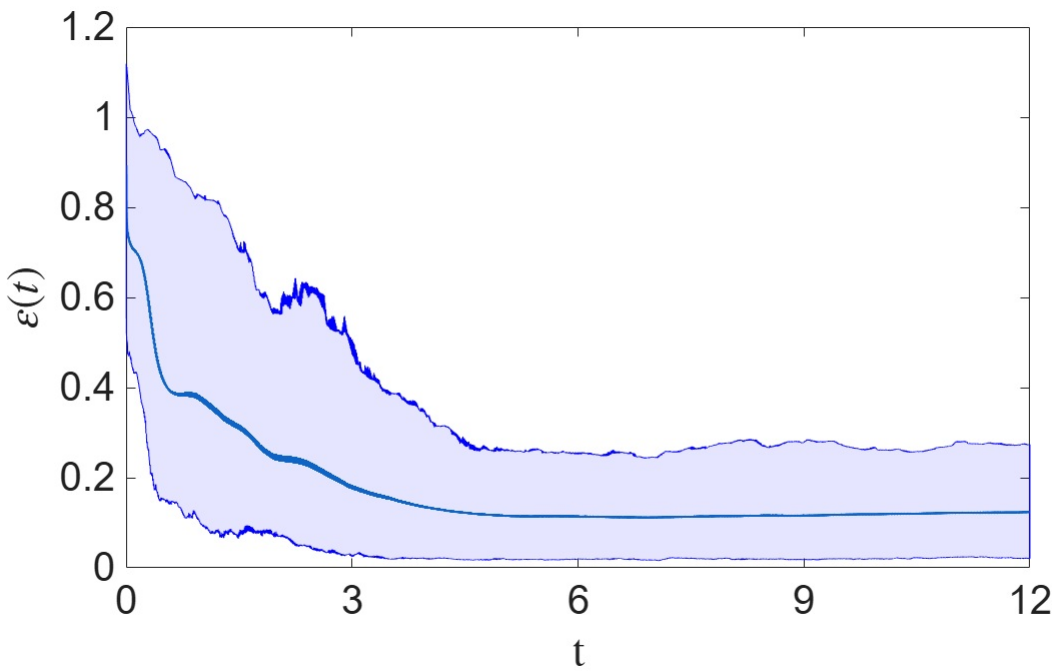} }
	\caption{The average and 90\% confidence intervals of the normalized matching errors $\varepsilon(t)$ for $ r(t)=[0.1,0.1]^\top $ under different noise levels ($\sigma=0.6,0.2,0.06$) and data collection (offline or online).}
	\label{Fdcon}
\end{figure*}

\begin{figure*}[b]
	\centering
	\subfigure[Offline data and $ \sigma=0.06 $]
	{\label{econoff003}
		\includegraphics[trim=20bp 255bp 45bp 270bp, clip, height=3.33cm]{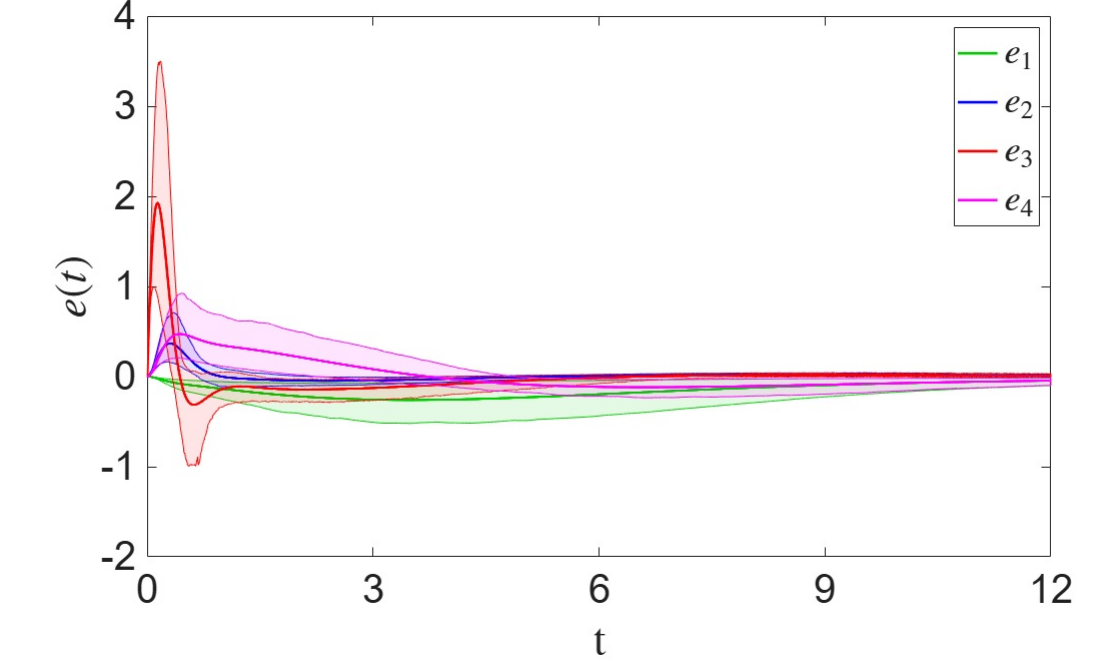} }
	\subfigure[Offline data and $ \sigma=0.2 $]
	{\label{econoff03}
		\includegraphics[trim=20bp 255bp 45bp 270bp, clip, height=3.33cm]{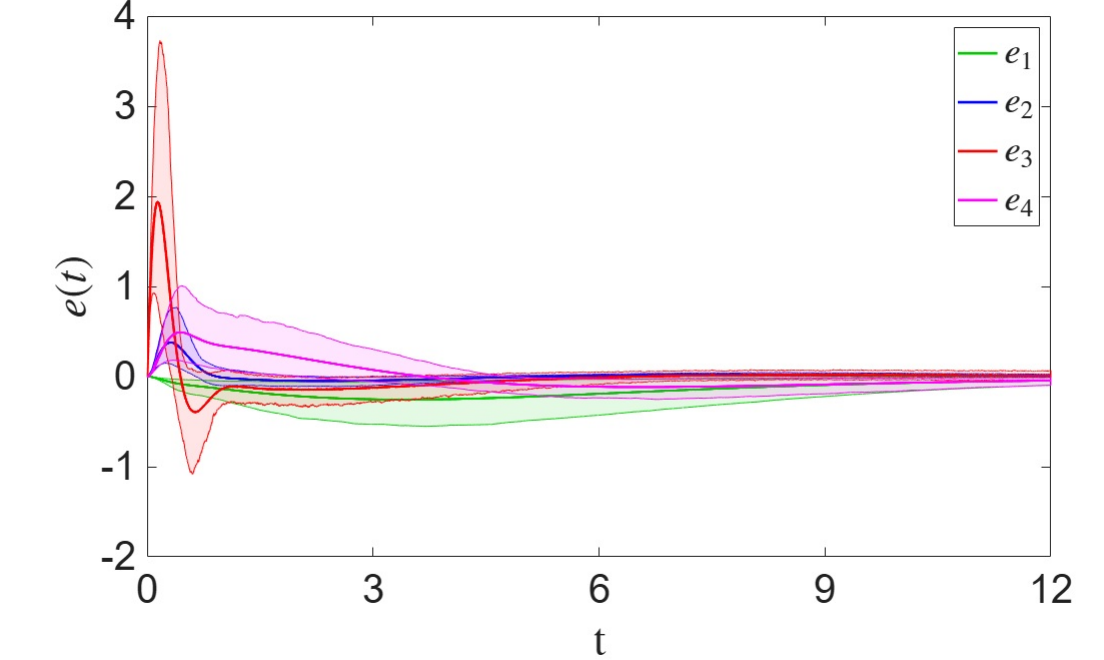} }
	\subfigure[Offline data and $ \sigma=0.6 $]
	{\label{econoff3}
		\includegraphics[trim=20bp 255bp 45bp 270bp, clip, height=3.33cm]{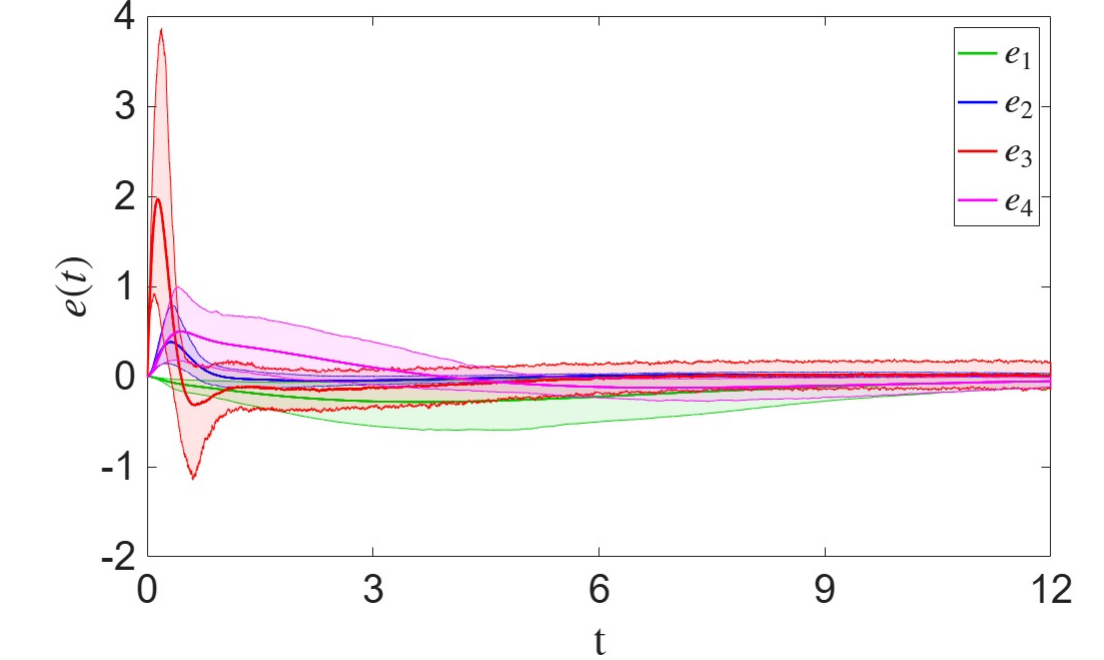} }
	\subfigure[Online data and $ \sigma=0.06 $]
	{\label{econ003}
		\includegraphics[trim=20bp 255bp 45bp 270bp, clip, height=3.33cm]{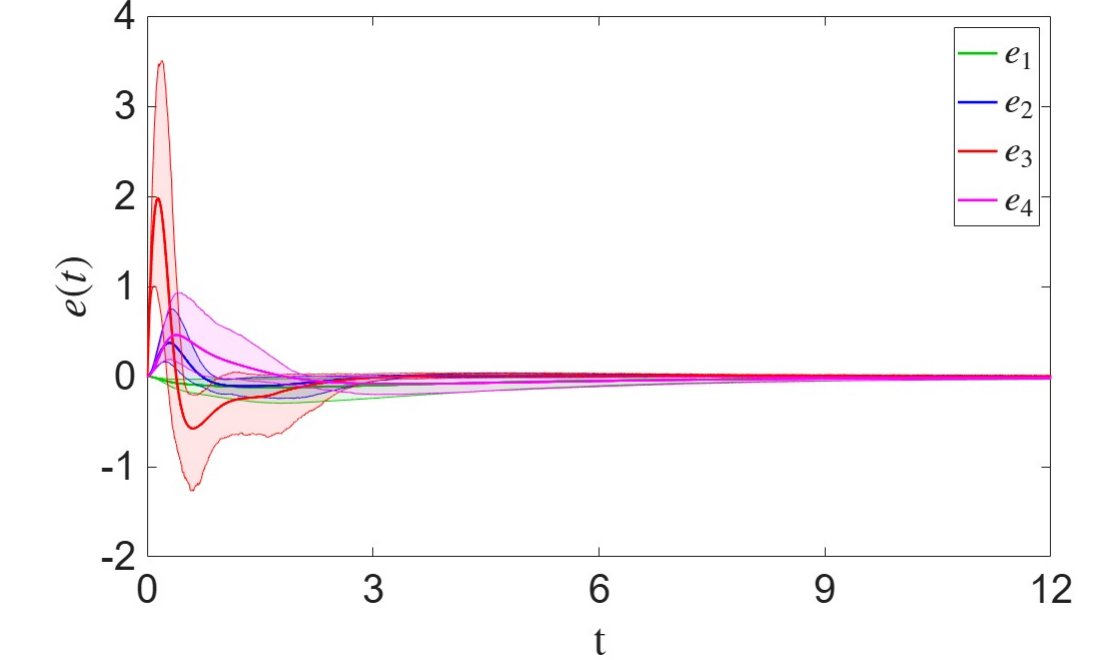} }
	\subfigure[Online data and $ \sigma=0.2 $]
	{\label{econ03}
		\includegraphics[trim=20bp 255bp 45bp 270bp, clip, height=3.33cm]{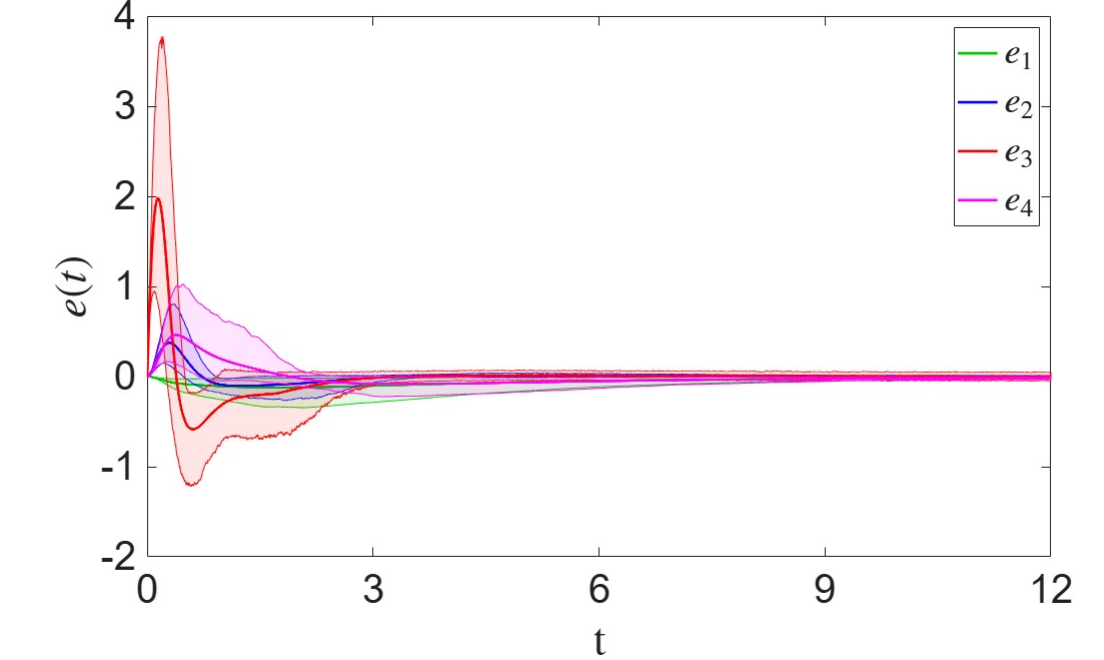} }
	\subfigure[Online data and $ \sigma=0.6 $]
	{\label{econ3}
		\includegraphics[trim=20bp 255bp 45bp 270bp, clip, height=3.33cm]{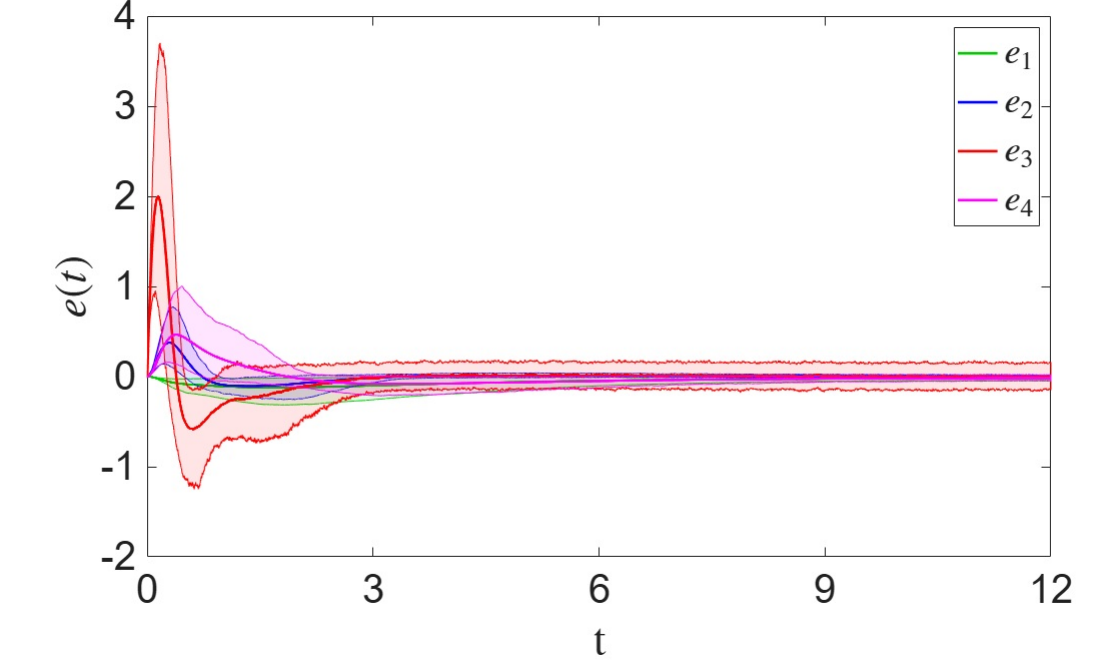} }
	\caption{The average and 90\% confidence intervals of the tracking errors $e(t)$ for $ r(t)=[0.1,0.1]^\top $ under different noise levels ($\sigma=0.06,0.2,0.6$) and data collection (offline or online).}
	\label{Fecon}
\end{figure*}

In this section, we illustrate the theoretical results using an aircraft longitudinal dynamics, describing the velocity along the $ x $- and $ z $-body-axis, the angular velocity around the $y$-body-axis, and the pitch angle \cite{moustakis2018adaptive}.
The dynamics are as in \eqref{cs} with
\begin{equation}\label{AB}
\begin{aligned}
A_\textrm{s}=&\begin{bmatrix}
-0.0190 &  0.0825 & -0.1005 & -0.3206 \\
-0.2154 & -2.7859 & 1.2031 & -0.0271 \\
3.2527 & -30.7871 & -3.5418 &  0      \\
0       &  0      & 1      &  0
\end{bmatrix},\\
B_\textrm{s}=&\begin{bmatrix}
0.0065 & 0.0534 \\
-0.6103 & 0.0020 \\
-74.6355 &  0.5431 \\
0 & 0
\end{bmatrix}\!, \quad
E_\textrm{s}=  \begin{bmatrix}
	0.001 &  0 & 0 \\
	0 & 0.01 & 0 \\
	0 & 0 & 1   \\
	0 & 0 & 0
\end{bmatrix},\\
\end{aligned}
\end{equation}
where the structure of $ E_\textrm{s} $ is motivated by the physics of the system: in particular,
there is no noise on the fourth dynamical equation because the time derivative of the pitch angle corresponds to the angular velocity around the $y$-body-axis.
Consider the reference model \eqref{cr} with
\begin{equation}\label{AmBm}
\begin{aligned}
A_\textrm{m}=&\begin{bmatrix}
-0.0215 &  0.0810 &  -0.0988 &  -0.3180 \\
-0.0706 & -2.6377 &  1.0345 &  -0.2636 \\
20.9585 & -12.6579 & -24.1637 &  -28.9269 \\
0      &  0      &  1     &   0
\end{bmatrix},\\
B_\textrm{m}=&\begin{bmatrix}
0.0066 & 0.0539 \\
-0.6167 & 0.0029 \\
-75.4185 &  0.6600 \\
0 & 0
\end{bmatrix}.
\end{aligned}
\end{equation}
For numerical integration of the continuous-time system \eqref{cs}, the reference model \eqref{cr}, and the filtered dynamics \eqref{xf}--\eqref{wf}, we employ an integration step $dt=0.001$s with zero-order-hold.
The process noise signal in \eqref{cs} is taken as white noise with seven noise levels
$
\sigma=0.02,0.06,0.2,0.6,2,6,20.
$
At every integration step, the discretized process noise is generated from the normal distribution
$
\mathcal N
\bigl(
0,\,
dt\,\sigma^2 I
\bigr).
$
The data matrices $ \left( U,X,X^\textrm{D}\right) $ are collected over 500 independent runs, each one lasting $3$ seconds starting from random initial conditions and applying random input.
After this offline data collection, we set the initial reference and initial system states as $ x_\textrm{m}(0)=x(0)=[2\ \!-\!1\ 1\ 0.5]^\top $, the same initial conditions as in \cite{moustakis2018adaptive}, and we set $M=1200$ and
$\left[ 
	t_1\ t_2\ \cdots\ t_{1200}
\right] =
\left[ 
	0.01\textrm{s}\ 0.02\textrm{s}\ \cdots\ 12\textrm{s}
\right]$.


We control the system using the adaptive law \eqref{alp} and the input \eqref{input} under two types of reference inputs:
\begin{itemize}
	\item[-] \underline{sinusoidal reference}: $ r(t)=\left[ \sin(t),\cos(t)\right]^\top $;
	\item[-] \underline{constant reference}: $ r(t)=\left[ 0.1,0.1\right]^\top $.
\end{itemize}
For each run, we have verified that condition \eqref{image} in Theorem \ref{Tec} is satisfied.
In order to verify the conditions in Theorem \ref{Tstability}, we proceed as follows. 
For each run, the smallest $\gamma$ satisfying \eqref{SNR} is computed for $M=0$ (offline data) and $M=1200$ (offline and online data) under both reference inputs. 
Such $\gamma$ is then used to determine whether \eqref{stability} is satisfied. 
We note that we have verified that, for all values of $\gamma$ satisfying \eqref{stability}$,$ the matrix $\Theta(Q_\gamma,A_{\rm m})$ has no eigenvalues on the imaginary axis. 
The percentage of runs in which the conditions of Theorem \ref{Tstability} are satisfied is reported in the table at the bottom of Page 12, showing that the percentage decreases as the noise level $\sigma$ increases. 
Interestingly, the sinusoidal reference input results in higher percentages, which can be explained with the excitation generated by sinusoidal signals. 
To verify Theorem \ref{Tbound}, we approximate the closed-loop system matrix $ A_\textrm{s}+B_{\hat{K}} $ with $ A_\textrm{s}+B_\textrm{s}\hat{K}(t) $ evaluated at $t=30$s to allow sufficient convergence.
The same table reports the percentage of runs in which $ A_\textrm{s}+B_\textrm{s}\hat{K}(t) $ is Hurwitz, which also decreases as the noise level $\sigma$ increases.
One may observe that such percentage remains at 100\% up to $\sigma=0.6$, even when condition \eqref{stability} is not satisfied.
This can be explained by noting that $\Lambda(Q_\gamma,A_\textnormal{m},0)<0$ guarantees the Hurwitz property in the entire set $\mathcal{Z}_{n}(\Pi(Q_\gamma,A_\textrm{m}))$ containing \emph{all} systems consistent with the collected data.
Hence, when \eqref{stability} does not hold, although the Hurwitz property in the entire $\mathcal{Z}_{n}(\Pi(Q_\gamma,A_\textrm{m}))$ cannot be guaranteed, it is still possible with high percentage that $ A_\textrm{s}+B_{\hat{K}} $ is Hurwitz for the true (but unknown) system.

To further illustrate parameter convergence, both offline and online data collection mechanisms are considered, corresponding to $M=0$ and $M=1200$, respectively.
We consider the following normalized matching error (calculated using the 2-norm) to quantify the parameter convergence
$$
\varepsilon(t):=
\frac
{
	\left\| \begin{bmatrix}
		A_\textrm{m}&B_\textrm{m}
	\end{bmatrix}-
	\begin{bmatrix}
		A_\textrm{s}+B_\textrm{s}\hat{K}(t)&B_\textrm{s}\hat{L}(t)
	\end{bmatrix}\right\| 
}
{\left\| \begin{bmatrix}
		A_\textrm{m}&B_\textrm{m}
	\end{bmatrix}\right\| }.
$$
Figs. \ref{Fdsin}-\ref{Fecon} report the results for each combination of reference input and data collection mechanism under $\sigma=0.06$, $0.2$ and $0.6$. 
The results come from 500 independent trajectories with different noise realizations.

For the sinusoidal reference, Fig.~\ref{Fdsin} presents the average and 90\% confidence intervals of the normalized matching errors $ \varepsilon(t) $, while Fig.~\ref{Fesin} illustrates the average and 90\% confidence intervals of the tracking error dynamics for different values of $ \sigma $.
Similarly, for the constant reference, Fig.~\ref{Fdcon} depicts the average and 90\% confidence intervals of $ \varepsilon(t) $, and Fig.~\ref{Fecon} shows the average and 90\% confidence intervals of $e(t)$ for different values of $ \sigma $.
Approximate matching and reference tracking can be observed.
Compared with offline data collection (Fig.~\ref{dsinoff3}-\ref{dsinoff003} and Fig.~\ref{dconoff3}-\ref{dconoff003}), online data collection (Fig.~\ref{dsin3}-\ref{dsin003} and Fig.~\ref{dcon3}-\ref{dcon003}) exhibits faster convergence. 
Interestingly, using online data does not necessarily provide a smaller limit matching error, as it can be seen in Fig.~\ref{dcon3}.
Furthermore, comparisons among $ \sigma=0.06$, $0.2$ and $0.6 $ reveal that the noise primarily affects the limit matching error, while having less impact on the convergence rate.


\section{Conclusion}

In this paper, we have proposed a robust MRAC framework that guarantees parameter convergence without requiring persistency of excitation of the involved signals.
Via a novel adaptive law, the closed-loop system has been proven to converge exponentially to an approximate solution to the matching equation, with the matching error being characterized as a function of the noise. 
Moreover, by combining the KYP lemma with QMI techniques, we have derived a necessary and sufficient condition under which the closed-loop system matrix is guaranteed to be Hurwitz in the limit.
In the absence of noise, the proposed framework reduces to a weaker data condition for exact parameter convergence when compared to data conditions for existing noise-free methods.
Future work will consider extensions to input-output settings.

\bibliographystyle{abbrv}       
\bibliography{References}           

%

\end{document}